\documentclass[reqno,10pt]{amsart}

\usepackage{amssymb,amsfonts,amsmath,amsthm,bm}
\usepackage[mathscr]{eucal}
\usepackage{bbm}
\usepackage{array}
\usepackage{mathdots}
\usepackage{ytableau}

\usepackage{hyperref}
\newcommand\myshade{85}
\colorlet{mylinkcolor}{red}
\colorlet{mycitecolor}{blue}
\hypersetup{
	linkcolor  = mylinkcolor,
	citecolor  = mycitecolor,
	urlcolor   = myurlcolor!\myshade!black,
	colorlinks = true,
}

\usepackage{tikz}

\usetikzlibrary{arrows,math,calc}
\usepackage[all,cmtip]{xy}

\usepackage{microtype}

\usepackage{enumerate}
\usepackage[shortlabels]{enumitem}

\usepackage{etoolbox}
\patchcmd{\section}{\scshape}{\bfseries}{}{}
\makeatletter
\renewcommand{\@secnumfont}{\bfseries}
\makeatother

\theoremstyle{definition}
\newtheorem{thm}{Theorem}
\newtheorem*{thm*}{Theorem}
\newtheorem{prop}[thm]{Proposition} 
\newtheorem{lem}[thm]{Lemma}
\newtheorem{defi}[thm]{Definition} 
\newtheorem{cor}[thm]{Corollary}
\newtheorem{rem}[thm]{Remark}

\newtheorem*{quest*}{Question}

\newtheorem{conj}[thm]{Conjecture}

\newcommand{\la}[1]{\mathfrak{#1}}

\newcommand{\ZZ}{\mathbb{Z}}

\newcommand{\QQ}{\mathbb{Q}}
\newcommand{\CC}{\mathbb{C}}

\newcommand{\sW}{\mathscr{W}}
\newcommand{\sD}{\mathscr{D}}

\DeclareMathOperator{\ch}{ch}

\DeclareMathOperator{\coeff}{Coeff}

\DeclareMathOperator{\qdim}{qdim}
\DeclareMathOperator{\wt}{wt}

\DeclareMathOperator{\im}{\mathrm{im}}

\DeclareMathOperator{\len}{\mathrm{len}}

\newcommand{\lietype}[1]{\mathrm{#1}}

\newcommand{\jones}{\mathbb{J}}

\newcommand*\circled[1]{\tikz[baseline=(char.base)]{
		\node[shape=circle,draw,inner sep=1pt] (char) {\small #1};}}

\allowdisplaybreaks

\begin{document}
	
\date{}

\title[]{Coloured invariants of torus knots, $\mathscr{W}$ algebras, and relative asymptotic weight multiplicities}
\author{Shashank Kanade}
\address{Department of Mathematics, University of Denver, Denver, CO 80208}
\email{shashank.kanade@du.edu}

\thanks{
Support from Simons Foundation's Travel Support for Mathematicians
\#636937 is gratefully acknowledged.
I am extremely grateful to Tomoyuki Arakawa, Arvind Ayyer, Thomas Creutzig, Shrawan Kumar, Andy Linshaw, and
Antun Milas for many stimulating conversations.
}

\begin{abstract}	
	We study coloured invariants of torus knots $T(p,p')$ (where $p,p'$ are coprime positive integers).
	When the colouring Lie algebra is simply-laced, and when $p,p'\geq h^\vee$, we use the representation theory of the corresponding principal affine $\mathscr{W}$ algebras to understand the trailing monomials of the coloured invariants. In these cases, we show that the appropriate limits of the renormalized invariants are equal to the characters of certain $\mathscr{W}$ algebra modules
	(up to some factors).
	This result on limits rests on a purely Lie-algebraic conjecture on asymptotic weight multiplicities which we verify in some examples.
\end{abstract}


\maketitle

\section{Introduction}

There are three interwoven strands in the present article -- knot theory, vertex operator algebras, and the behavior of weight multiplicities in modules for Lie algebras.

Specifically, we wish to relate coloured invariants of torus knots with characters of principal $\sW$ algebras which are certain kinds of vertex operator algebras (VOAs). 
Surprisingly, this relation rests on a conjecture about relative asymptotic weight multiplicities in finite-dimensional irreducible modules of finite-dimensional simple Lie algebras. We will get to this conjecture in due course, but interested readers may already take a peek at Definition \ref{def:relasympmul} and Conjecture \ref{conj:main}; both requiring only elementary notions.

The setup is as follows. 

We start with a fixed knot, say $K$, and a finite-dimensional simple Lie algebra $\overline{\la{g}}$
(most notions for $\overline{\la{g}}$ will be denoted with an overline, and the notions for the corresponding untwisted affinized Lie algebra $\la{g}$ without an overline). A dominant integral weight $\lambda\in \overline P_+$ gives rise to a finite-dimensional irreducible $\overline{\la{g}}$ module $L(\lambda)$. We ``colour'' the knot $K$ with $L(\lambda)$ and this data determines the coloured invariant $\jones_{K}(\lambda)$, which is a Laurent series in some fractional power of a formal variable $q$.
For this paper, we assume for convenience that $\jones$ is framing dependent and un-normalized (i.e., on unknot with zero framing, it evaluates to the $q$-character of the colouring module $L(\lambda)$). 
Also, for this paper, we will only concentrate on torus knots $K=T(p,p')$ where $p,p'\geq 1$ are coprime positive integers.

Importantly, we now renormalize $\jones$ by dividing through with its trailing monomial to obtain $\widehat\jones$.

The main question then is:
\begin{quest*}
	What can we say about:
	\begin{align*}
		\lim_{\substack{n\rightarrow \infty \\ n\lambda\in \lambda+\overline{Q}}} \widehat{\jones}_{K}(n\lambda)?
	\end{align*}
	When does the limit exist? Is it a ``nice'' $q$-series?
\end{quest*}

In many instances, this question is known to have very beautiful answers. For example, Armond--Dasbach \cite{ArmDas-RR} proved 
one of the celebrated Rogers--Ramanujan identities \cite{And-book} completely knot-theoretically using $K=T(2,5)$, $\overline {\la{g}}=\la{sl}_2$ by evaluating the limit above in two different ways. 

As one might expect, the $\la{sl}_2$ case has been studied extensively, with the most comprehensive results for the limits given by Garoufalidis--L\^{e} \cite{GarLe-Nahm}. Some general results on the $\la{sl}_3$ case are also now available due to the results of Yuasa \cite{Yua-stability}.
Meanwhile, Garoufalidis--Vuong studied \cite{GarVuo} these limits for torus knots when $\overline{\la{g}}$ has rank $2$.

Previously, in \cite{Kan-torus} we studied the case of torus knots coloured with the irreducible modules $\mathrm{Sym}^n(\CC^r)$ of $\la{sl}_r$, and showed how the limits are  related to the principal $\sW$ algebra characters. For similar results on certain torus links, relating the limits to characters of logarithmic VOAs, see \cite{Kan-toruslinks} and also the work of Hikami--Sugimoto \cite{HikSug}.

The aim of this paper is to generalize our results in \cite{Kan-torus} to general (finite-dimensional, simple) Lie algebras $\overline{\la{g}}$. We have not fully succeeded in this aim, the main obstacle being the relative asymptotic multiplicity conjecture that we have alluded to above. Nonetheless, our results go beyond the results of Garoufalidis--Vuong \cite{GarVuo} in that we are able to clarify the resulting limits.

The main knot-theoretic tool we use to evaluate these invariants for torus knots is the Rosso--Jones formula \cite{RosJon-torus} as explained in the work of Morton \cite{Mor-coloured} and used in the papers \cite{GarMorVuo}, \cite{GarVuo}, etc. This paper of Morton is central to our argument -- in fact, our treatment of this part is  a straight-forward generalization of the one given by Morton.

Using the Rosso--Jones formula, one immediately sees that the coloured invariants $\jones$ of torus knots look a lot like finitizations of characters of principal $\sW$ algebras, at least in the simply-laced case. Namely, there is a sum over a finite subset of the extended affine Weyl group, except, importantly, the $\jones$ involves  weight multiplicities whereas the corresponding factors in the character formula for the $\sW$ algebra modules are all $1$; compare equations \eqref{eqn:jones} and \eqref{eqn:wchar3} below.

At this point, our path diverges from the one taken by Garoufalidis--Vuong \cite{GarVuo}. Whereas they keep the various weight multiplicities untouched and only shift the power of $q$ to start at $q^0$ (see their notion of $c$-stability \cite[Def.\ 1.3]{GarVuo}), 
we renormalize by dividing through by the trailing monomial in $\jones$ to obtain $\widehat{\jones}$. The reason is that ultimately limits of $\widehat{\jones}$ are related to VOA characters. This is the first key insight. 

Naturally, the question is, what is the trailing monomial? This, in general, appears to not be easy. 

In the simply-laced case and when $p,p'\geq h^\vee$ (the dual Coxeter number of $\overline{\la{g}}$), we are thankfully able to use  powerful theorems in the representation theory of $\sW$ algebras due to Arakawa  \cite{Ara-vanishing}, \cite{Ara-repW1},
\cite{Ara-BGG}, \cite{Ara-c2}, \cite{Ara-princrat} to tackle this question. The main idea is to use resolutions of $\sW$ algebra modules in terms of Verma modules to understand the lowest order term of $\jones$. This is the second key insight.

In the non-simply-laced cases, our answer is much less satisfactory  since the theory of $\sW$ algebras is not applicable here; see Remark \ref{rem:nonSClimit} below.

In any case, after finding the correct trailing monomial, we renormalize $\jones$. Resulting formulas for $\widehat\jones$ are captured in Theorems \ref{thm:ADEjones_hat} and \ref{thm:minexp}.
Now, the coefficients of $\widehat{\jones}$ involve quotients of weight multiplicities of the module $L(\lambda)$. As we consider $n\lambda$ with $n\rightarrow \infty$, these quotients become \emph{relative asymptotic multiplicities}. To eventually match with $\sW$ algebra characters (in the simply-laced case), it would be sufficient for these quotients to all be $1$ in the limit.
This is precisely our relative asymptotic multiplicity conjecture, encapsulated as Conjecture \ref{conj:main}.
Note that this purely Lie-algebraic conjecture is not restricted to the simply-laced types, and we speculate that it holds in general.
This is our third contribution in this paper.

Assuming Conjecture \ref{conj:main}, our main theorem, where we calculate the limits of $\widehat{\jones}$ and in the simply-laced cases relate them to characters of relevant $\sW$ algebra modules, is given as Theorem \ref{thm:ADElimits}.

Clearly, this multiplicity conjecture is of independent interest, and we verify it in some examples included in Section \ref{sec:relasympmul}.

We will begin this paper with an extensive study of $\sW$ algebra characters. Naturally, we will first spend some effort in setting up the notation regarding affine Lie algebras. As the paper progresses, many unanswered questions will emerge, some of which are collected in the very last section.

\section{Preliminaries on affine Lie algebras}
Our main reference for this section is Wakimoto's book \cite{Wak-book2}, and also the papers \cite{FreKacWak} and \cite{KacWak-modinv}.

\subsection{Triangular decomposition}
Let $\la{g}$ be an (untwisted) affine Lie algebra of type $X_\ell^{(1)}$ (the simply-laced case is when $X\in \{\lietype{A},\lietype{D},\lietype{E}\}$).
We realize $\la{g}$ as
\begin{align*}
	\la{g}=\overline{\la{g}}\otimes \CC[t,t^{-1}] \oplus \CC c \oplus \CC d,
\end{align*}
where $\overline{\la{g}}$ is the corresponding  finite-dimensional simple Lie algebra of type $X_\ell$ over $\CC$. 
We will often use $\overline{\la{g}}$ and the subspace $\overline{\la{g}}\otimes t^{0} \subset \la{g}$ interchangeably.

We have the usual triangular decomposition for $\la{g}$, using the triangular decomposition for $\overline{\la{g}}=\overline{\la{n}}_-\oplus \overline{\la{h}}\oplus \overline{\la{n}}_+$:
\begin{align}
	\la{g}=\la{g}_- \oplus \la{h}\oplus \la{g}_+,
\end{align}
with 
\begin{align}
	\la{g}_-&=
	\overline{\la{n}}_-\otimes \CC[t^{-1}]\,\,\oplus\,\,
	\overline{\la{h}}\otimes t^{-1}\CC[t^{-1}]\,\,\oplus\,\, 
	\overline{\la{n}}_+\otimes t^{-1}\CC[t^{-1}]\\
	\la{g}_+&=
	\overline{\la{n}}_+\otimes \CC[t]\,\,\oplus\,\,
	\overline{\la{h}}\otimes t\CC[t]\,\,\oplus\,\, 
	\overline{\la{n}}_-\otimes t\CC[t]\\
	\la{h}&=\overline{\la{h}}\oplus \CC c\oplus \CC d.
\end{align}	

\subsection{Bases and pairings}
We will let $\mathsf{C}=(\mathsf{C}_{ij})_{0\leq i,j\leq \ell}$ be the Cartan matrix of $\la{g}$. The Cartan matrix of $\overline{\la{g}}$ is naturally $(\mathsf{C}_{ij})_{1\leq i,j\leq \ell}$. The labels and co-lables corresponding to $\mathsf{C}$ are denoted as $a_0,\cdots, a_\ell$ and $a_0^\vee,\cdots, a_\ell^\vee$.
The Coxeter and the dual Coxeter numbers are given as:
\begin{align*}
	h=a_0+\cdots+a_\ell,\quad h^\vee = a_0^{\vee} + \cdots +a_\ell^\vee.
\end{align*}
We always have $a_0^\vee=1$. We also have $a_0=1$ unless $\la{g}$ is of type $\lietype{A}_{2n}^{(2)}$, however, we have restricted our attention to untwisted affine Lie algebras, so this case is excluded.
	
A basis for $\la{h}$ is $\{\alpha_1^\vee,\cdots, \alpha_\ell^\vee,c,d\}$. Another natural basis is $\{\alpha_0^\vee,\alpha_1^\vee,\cdots, \alpha_\ell^\vee,d\}$,
where we have the relation 
$$c=a_0^\vee\alpha_0^\vee +\cdots+ a_\ell^\vee\alpha_\ell^\vee.$$
Several natural bases for $\la{h}^*$ are given as
$\{\Lambda_0,\alpha_0,\cdots,\alpha_\ell\}$, or $\{\Lambda_0,\delta,\alpha_1,\cdots,\alpha_\ell\}$, or 
$\{\delta,\Lambda_0,\cdots,\Lambda_\ell\}$.
We have 
$$\delta = a_0\alpha_0+\cdots+a_\ell\alpha_\ell.$$
The natural pairing between $\la{h}$ and $\la{h}^*$ is denoted by $\langle\cdot,\cdot\rangle$ and it satisfies $(0\leq i,j\leq \ell)$:
\begin{align*}
	\langle \Lambda_i,d\rangle =0,&\quad 
	\langle \Lambda_i,\alpha_j^\vee\rangle =\delta_{ij}\\
	\langle \alpha_j, \alpha_i^\vee\rangle = \mathsf{C}_{ij},&\quad 
	\langle \alpha_j, d\rangle = \delta_{j0}\\
	\langle \delta, \alpha_i^\vee\rangle &= \langle \alpha_i, c\rangle = 0.
\end{align*}
Here, and below, $\delta_{ij}$ is the Kronecker delta.
We have that:
\begin{align*}
	\overline{\la{h}}=\mathrm{Span}_{\CC}\{\alpha_1^\vee,\cdots,\alpha_\ell^\vee\},
\end{align*}
and we put
\begin{align*}
	\overline{\la{h}}^*=\mathrm{Span}_{\CC}\{\alpha_1,\cdots,\alpha_\ell\}.
\end{align*}
Under the decomposition
\begin{align*}
	\la{h}^*=\CC\Lambda_0 \oplus \overline{\la{h}}^* \oplus \CC\delta,
\end{align*}
the component in $\overline{\la{h}}^*$ of $\lambda\in \la{h}^*$ is denoted as $\overline{\lambda}$ and called the \emph{classical} or the \emph{finite} part of $\lambda$.

Given $\lambda\in \la{h}^*$, the number $\langle \lambda, c\rangle$ is called its level.
The set of elements of level $k$ will be denoted by $\la{h}^*_k$.

\subsection{Bilinear form}
There is a non-degenerate symmetric bilinear form on $\la{h}$ determined by:
\begin{align*}
	(\alpha_i^\vee,\alpha_j^\vee)=\mathsf{C}_{ij}\dfrac{a_j}{a_j^\vee},
	\quad (\alpha_i^\vee,d)=a_0\delta_{i0}=\delta_{i0},\quad (d,d)=0,
\end{align*}
Note that we have:
\begin{align*}
	(c,c) =0.
\end{align*}
It is customary to identify $\la{h}$ and $\la{h}^*$ under this form $( \cdot,\cdot )$, and we shall do so.
Under this identification, we have for all $0\leq i\leq \ell$,
\begin{align*}
	a_i^\vee\alpha_i^\vee \longleftrightarrow a_i\alpha_i,\quad d\longleftrightarrow a_0\Lambda_0=\Lambda_0,\quad c \longleftrightarrow \delta.
\end{align*}	
We may now transfer the form $(\cdot,\cdot)$ to $\la{h}^*$, and under this form, we have
\begin{align*}
	( \Lambda_0,\Lambda_0)=(\delta,\delta)=0,\quad (\Lambda_0,\alpha_i)=\dfrac{a_0^\vee}{a_0}\delta_{i0}=\delta_{i0}.
\end{align*}
As usual, we denote $( \lambda,\lambda)$ by $\|\lambda\|^2$ for $\lambda\in \la{h}^*$.

\subsection{Roots, weights, etc.}
We now recall various notations regarding root systems, root and weight lattices, etc.

We let the positive, negative, and all the roots for $\overline{\la{g}}$ be denoted by  $\overline{\Delta}_+,\overline{\Delta}_-,\overline{\Delta}=\overline{\Delta}_+\cup \overline{\Delta}_-$, respectively. Let $\overline{\Pi}$ denote the set of positive simple roots $\{\alpha_1,\cdots,\alpha_\ell\}$, and $\overline{\Pi}^\vee$ are the corresponding co-roots $\{\alpha_1^\vee,\cdots,\alpha_\ell^\vee\}$. Analogous notions for $\la{g}$ are denoted without the
overline.
We have the Weyl vector $\rho$ and Weyl co-vector $\rho^\vee$ to be any vectors which satisfy:
\begin{align*}
	\langle \rho, \alpha_i^\vee\rangle = 1, \quad 
	( \rho^\vee, \alpha_i ) = 1, \quad (0\leq i\leq \ell).
\end{align*}
We may take
\begin{align*}
	\rho = h^\vee \Lambda_0 + \overline{\rho},
\end{align*}
where
\begin{align*}
	\overline{\rho} = \frac{1}{2}\sum_{\alpha\in \overline{\Delta}_+} \alpha.
\end{align*}
	
We now have the following lattices
\begin{align*}
	Q=\ZZ\alpha_0+\cdots+\ZZ\alpha_\ell,&\quad \overline{Q}=\ZZ\alpha_1+\cdots+\ZZ\alpha_\ell,	\\
	Q^\vee=\ZZ\alpha_0^\vee+\cdots+\ZZ\alpha_\ell^\vee,&\quad \overline{Q}^\vee=\ZZ\alpha_1^\vee+\cdots+\ZZ\alpha_\ell^\vee,
\end{align*}
and the weight lattice for $\overline{\la{g}}$
\begin{align*}
	\overline{P}&=\ZZ\overline{\Lambda}_1+\cdots +\ZZ\overline{\Lambda}_\ell.
\end{align*}
The set of dominant integral weights for $\overline{\la{g}}$ is
\begin{align*}
	\overline{P}_+=\{ \lambda\in \overline{P}\,|\, \langle \lambda, \alpha_i^\vee\rangle \in \ZZ_{\geq 0}, (1\leq i\leq \ell)\}.
\end{align*}
We have the set of weights for $\la{g}$ as
\begin{align*}
	P = \{\lambda\in \la{h}^* \,\vert\, \langle \lambda, \alpha_i^\vee\rangle \in \ZZ, (0\leq i\leq \ell)\}.
\end{align*}
The subset $P_+$ is defined by replacing $\ZZ$ with $\ZZ_{\geq 0}$, and the subset $P^k_+$ further restricts to $\lambda$ of level $k$. The sets $P^\vee, P^{\vee}_+, P^{\vee,k}_{+}$ are defined by replacing $\alpha_i^\vee$ with $\alpha_i$ and $\langle\cdot,\cdot\rangle$ with $(\cdot,\cdot)$ appropriately.

\subsection{Weyl group}
The elements $r_i \in GL(\la{h}^*)$ for $0\leq i\leq \ell$ defined by
\begin{align*}
	r_i(\lambda)  = \lambda - \langle \lambda,\alpha_i^\vee\rangle \alpha_i
\end{align*}
generate the (affine) Weyl group $W$ of $\la{g}$. The subgroup $\overline{W}$ generated by $r_1,\cdots,r_\ell$ is isomorphic to the Weyl group of $\overline{\la{g}}$.
For all $\overline{w}\in \overline{W}$ and $w\in W$, we have:
\begin{align*}
	\overline{w}\Lambda_0 = \Lambda_0,\quad \overline{w}\overline{\la{h}}^*=\overline{\la{h}}^*,\quad w\delta = \delta.
\end{align*}
Consequently, for $\lambda\in \la{h}^*$ and $\overline{w}\in\overline{W}$,
\begin{align*}
	\overline{\overline{w}\lambda}=\overline{w}\,\overline{\lambda}.
\end{align*}
The shifted action of the Weyl group is defined for $w\in W$, $\lambda\in \la{h}^*$ by
\begin{align*}
	w\circ \lambda = w(\lambda+\rho) - \rho.
\end{align*}
The groups $W$ and $\overline{W}$ have a length function such that the simple reflections have length $1$. We then have the sign representation
\begin{align*}
	\varepsilon(w) = (-1)^{\len(w)} = \det{}_{\la{h}^*}(w)
\end{align*}
for $w\in W$.
The longest element of $\overline{W}$ is denoted as $w_0$. It satisfies
\begin{align*}
	\len(w_0) = |\overline{\Delta}_+|,\quad w_0\overline{\rho}=-\overline{\rho}.
\end{align*}

The Weyl group $W$ preserves the form $(\cdot,\cdot)$ on $\la{h}^*$.
The real roots are:
\begin{align*}
	\Delta^{re}=W\Pi.
\end{align*}	
For $\alpha\in \Delta^{re}$, $(\alpha,\alpha )\neq 0$ and  we define:
\begin{align*}
	\alpha^\vee = \frac{2\alpha}{(\alpha,\alpha)}.
\end{align*}
the real co-roots are then
\begin{align*}
	\Delta^{\vee,re}=\{\alpha^{\vee}\,|\, \alpha\in \Delta^{re}\}=W\Pi^\vee.
\end{align*}
The positive real roots and coroots are denoted $\Delta^{re}_+$, $\Delta^{\vee,re}_+$, respectively.

For an element $\alpha\in \mathrm{Span}_{\CC}\{\alpha_1,\cdots,\alpha_\ell\}$, define the translations $t_\alpha\in GL(\la{h}^*)$ by
\begin{align*}
	t_\alpha(\lambda) = \lambda + (\lambda,\delta)\alpha- \left(\frac{1}{2}\|\alpha\|^2(\lambda,\delta) + (\lambda,\alpha)\right)\delta.
\end{align*}
Note that for all $\alpha\in \mathrm{Span}_{\CC}\{\alpha_1,\cdots,\alpha_\ell\}$, we have:
\begin{align*}
	t_\alpha(\delta)=\delta.
\end{align*}
Since we are working with untwisted affine Lie algebras, we have:
\begin{align}
	W\cong \overline{W}\ltimes t_{\overline{Q}^\vee}=\{wt_{\alpha}\,\vert\, w\in\overline{W}, \alpha\in \overline{Q}^\vee\}.
	\label{eqn:Wsemidirect}
\end{align}
(recall that we have identified $\la{h}$ and $\la{h}^*$, so $\overline{Q}^\vee$ needs to be transported to $\la{h}^*$ accordingly.)

We will also require the extended affine Weyl group $\widetilde{W}$, which is given as 
\begin{align*}
\widetilde{W}=\widetilde{W}_+\ltimes W,
\end{align*}
and we now describe the finite group $\widetilde{W}_+$.
Let 
\begin{align}
	J = \{ i\,|\, a_i=a_i^\vee=1, 0\leq i\leq \ell\}.
	\label{eqn:Jset}
\end{align}
Then, $\widetilde{W}_+$ is isomorphic to the subgroup of Dynkin diagram automorphisms of $\la{g}$ that acts simply transitively on $J$.
In particular, we have:
\begin{align*}
	|\widetilde{W}_+|=|J|.
\end{align*}
It is important to have a concrete realization of $\widetilde{W}_+$.
The element corresponding to $0\in J$ is the identity.
If $0\neq j\in J$, there is a unique element $\overline{\sigma}_j\in \overline{W}$ that satisfies:
\begin{align*}
	\overline{\sigma}_j(\overline{\Pi})=\{ \alpha_1,\cdots, \alpha_{j-1},\widehat{\alpha_j},\alpha_{j+1}\cdots, \alpha_\ell,-\theta\},
\end{align*}
where $\widehat{\cdot}$ denotes omission, and $\theta=\delta-\alpha_0$ is the highest root of $\overline{\Delta}$. 
We assign $\overline{\sigma}_0=1$.
Now the element of $\widetilde{W}_+$ corresponding to $j$ is given as $\sigma_j = t_{\overline{\Lambda}_j}\overline{\sigma}_j$.
It is known that for $j\in J$:
\begin{align}
	\sigma_j(\Lambda_0)&=\Lambda_j\,\,(\mathrm{mod}\,\CC\delta),\notag\\
	\overline{\sigma}_j^{-1}(\alpha_j)&=-\theta,\notag\\
	\varepsilon(\overline{\sigma}_j)=
	&(-1)^{\len(\overline\sigma_j)}=
	(-1)^{2( \overline{\Lambda}_j,\overline{\rho})},
	\label{eqn:sgnsigmaj}\\
	\sigma_j\rho &\equiv \rho \,\,(\mathrm{mod}\,\CC\delta),
	\notag
	\\
	\overline{\sigma}_j^{-1}\overline{\Lambda}_j&=-\overline{\Lambda}_i.
	\label{eqn:Jij}
\end{align}
The first three can be found as \cite[Lem.\ 1.2.3]{Wak-book2}, the fourth as \cite[(1.48)]{Wak-book2}. For the last one, $i$ is such that $\sigma_j=\sigma_i^{-1}$, see \cite[Lem.\ 1.2.6]{Wak-book2}; here we also have that $\overline\Lambda_j+\overline\Lambda_i\in \overline{Q}$.

\subsection{Admissible weights}
An element $\lambda\in \la{h}^*$ is called an \emph{admissible weight} (see \cite[Sec.\ 3.1]{Wak-book2}) if it satisfies the following two conditions:
\begin{enumerate}
	\item For all $\alpha\in \Delta_+^{\vee,re}$, $\langle \lambda+\rho,\alpha\rangle \not\in\{0,-1,-2,\cdots\}$,
	\item $\QQ$-span of $\{\alpha\in \Delta_+^{\vee,re}\,|\, \langle \lambda+\rho,\alpha\rangle \in \ZZ\}$ = $\QQ$-span of $\Delta_+^{\vee,re}$.
\end{enumerate}

A complex number $k$ is called \emph{admissible} if the weight $k\Lambda_0$ is admissible. It is known \cite{KacWak-modinv}, \cite{KacWak-rationality} that $k$ is admissible iff the following holds:
\begin{align*}
	k+h^\vee = \frac{p}{p'}\,\,\mathrm{with}\,\,p,p'\in \ZZ_{>0}, (p,p')=1, 
	p\geq 
	\begin{cases}
		h^\vee & \mathrm{if}\,\, (r^\vee,p')=1\\
		h & \mathrm{if}\,\, (r^\vee,p')=r^\vee
	\end{cases}
\end{align*}
where $r^\vee$ is the maximal number of edges in the Dynkin diagram of $\la{g}$.

We put:
\begin{align*}
	\Delta_{\lambda,+}^{\vee,re}=\{\alpha\in \Delta_+^{\vee,re}\,|\, \langle \lambda+\rho,\alpha\rangle \in \ZZ_{>0}\},
\end{align*}
and we let $\Pi^\vee_\lambda$ denote the set of simple elements in $\Delta_{\lambda,+}^{\vee,re}$ (i.e, elements which can not be written as sum of two elements in $\Delta_{\lambda,+}^{\vee,re}$).

We put 
\begin{align*}
	W_\lambda = \langle r_{\alpha}, \alpha\in \Pi_\lambda^\vee\rangle,
\end{align*}
the natural length function on $W_\lambda$ will be denoted by $\len_\lambda$.

We say that an admissible weight $\lambda$ is \emph{principal admissible} if in addition to the above two conditions, we have that $\Pi^\vee_\lambda$ is of the same type as $\Pi^\vee$. The set of principal admissible weights of level $k$ is denoted as $Pr^k$.

A principal admissible weight $\lambda$ is called \emph{non-degenerate} if (see \cite[Sec.\ 4.1]{Wak-book2})
\begin{align*}
	\langle \lambda+\rho, \alpha^\vee\rangle \not\in \ZZ,\,\, \mathrm{for \,\,all}\,\,\alpha^\vee\in \overline{\Delta}^\vee.
\end{align*}
Let us denote the set of non-degenerate principal admissible weights of level $k$ by $Pr^{k}_{nondeg}$.

We say that a level $k$ is \emph{non-degenerate principal admissible level} iff $Pr^{k}_{nondeg}\neq \emptyset$ which happens iff (see \cite[Eq.\ (323)]{Ara-repW1})
\begin{align*}
k+h^\vee = \frac{p}{p'}\,\, \mathrm{with}\,\,
p,p'\in \ZZ_{>0}, (p,p')=1, p\geq h^\vee, p'\geq h, (p',r^\vee)=1.
\end{align*}
For $\la{g}$ simply-laced, the condition $(p',r^\vee)=1$ is vacuous.

Let $k$ be a non-degenerate principal admissible number. For the characterization of the set $Pr^{k}_{nondeg}$, we refer to \cite[Prop.\ 4.1.2]{Wak-book2}, and it is given as follows:

For $(\nu, \mu)\in P_+^{p-h^\vee}\times P_+^{\vee, p'-h}$, define:
\begin{align*}
	\Lambda(\nu,\mu)
	&= \nu - (k+h^\vee)(\mu+\rho^\vee) + (k+h^\vee)\Lambda_0.
\end{align*}
Then, for all $\overline w\in \overline{W}$, $\overline w\circ \Lambda(\lambda,\mu)\in Pr^{k}_{nondeg}$ and conversely, every element of $Pr^{k}_{nondeg}$ can be written like this although not necessarily in a unique way.
In fact, for $(\nu,\mu,\overline w), (\nu',\mu',\overline w')\in P_+^{p-h^\vee}\times P_+^{\vee, p'-h}\times\overline{W}$,
\begin{align*}
	\overline w\circ \Lambda(\nu,\mu)=\overline w'\circ \Lambda(\nu',\mu') \Longleftrightarrow (\nu,\mu,\overline w)=(\sigma_j\nu',\sigma_j\mu',\overline w'(\overline{\sigma}_j)^{-1}),
\end{align*}
for some $j\in J$.

Combining \cite[Eq.\ (3.26), (3.13)]{Wak-book2}, the integral Weyl group with respect to such a $\overline w\circ\Lambda(\nu,\mu)$ is given by:
\begin{align}
	W_{\overline w\circ \Lambda(\nu,\mu)}=\overline wt_{-\overline{\mu+\rho^\vee}} \left(\overline{W}\ltimes t_{p' \overline{Q}^\vee}\right) t_{\overline{\mu+\rho^\vee}}\overline w^{-1}.
	\label{eqn:Wnumu}
\end{align}

\subsection{Simply-laced case} 
\label{sec:simplylaced}
When $\la{g}$ is simply-laced, or equivalently, when the Cartan matrix $\mathsf{C}$ is symmetric, all of which happens iff $\la{g}$ is of type $X_\ell^{(1)}$ with $X\in\{\lietype{A},\lietype{D}, \lietype{E}\}$, several formulas naturally simplify. Namely, we have for $0\leq i \leq \ell$,
\begin{align*}
	a_i &= a_i^{\vee},\quad  \alpha_i^\vee \longleftrightarrow \alpha_i,\quad h=h^\vee,\quad \rho = \rho^\vee,\quad \overline{Q}=\overline{Q}^\vee,\quad P^{\vee,p'-h}_+=P_+^{p'-h},
\end{align*}
etc.

\section{Principal $\sW$ algebras}
In this section, we study characters of (irreducible) modules of the principal affine $\sW$ algebras.
We will carefully understand the location of the lowest order terms in these characters, using the representation-theoretic results. 
In the simply-laced case, this will be used to understand the trailing monomials of $\jones$ invariants for the torus knots.
It will take us too far to define the principal affine $\sW$ algebras, which are certain vertex operator algebras; we only recall the important theorems of Arakawa on their representation theory.

Due to the triangular decomposition of $\la{g}$, given $\lambda\in \la{h}^*$, we have the Verma module $\mathbb{M}(\lambda)$ for $\la{g}$ with highest weight $\lambda$ and its unique irreducible quotient $\mathbb{L}(\lambda)$.
We also have natural notions of category $\mathcal{O}$ and $\mathcal{O}_k$ (of level $k$ modules in $\mathcal{O}$) for $\la{g}$. A certain full subcategory $\dot{\mathcal{O}}_k$ was introduced in \cite{Ara-repW1} and
$\mathbb{M}(\lambda)$ and $\mathbb{L}(\lambda)$ belong to $\dot{\mathcal{O}}_k$ whenever $\lambda$ has level $k$ \cite[Lem.\ 7.1.1]{Ara-repW1}. In some cases, we will need to denote $\mathcal{O}$ by $\mathcal{O}^{\la{g}}$ to emphasize the underlying $\la{g}$.

The universal principal affine $\sW$ algebra for $\la{g}$ at level $k$ is denoted by $\sW^{k}(\la{g})$ (or, when $\la{g}$ is clear, simply $\sW^k$). Its unique simple quotient is denoted by $\sW_k(\la{g})$. 
It is important to note that the notation in \cite{Ara-repW1} is different -- there, mainly the universal affine principal $\sW$ algebra is studied, and it is denoted by $\sW_k$. However, the notation we shall use is also fairly standard.
In any case, there is a notion of category $\mathcal{O}$ for $\sW^k$, denoted by $\mathcal{O}(\sW^k)$ \cite[p.\ 286]{Ara-repW1}.

We now describe the Verma modules and irreducible modules for $\sW^k$.

Given $\overline{\lambda}\in \overline{\la{h}}^*$, we have the $\overline{\la{g}}$ Verma module $\overline{\mathbb{M}}(\overline{\lambda})$. Let $\gamma_{\overline{\lambda}}$ be the central character corresponding to  $\overline{\mathbb{M}}(\overline{\lambda})$, i.e., for all $z\in \mathcal{Z}(\la{g})$ (the center of the universal enveloping algebra) and $m\in \overline{\mathbb{M}}(\overline{\lambda})$, we have $zm=\gamma_{\overline{\lambda}}(z)m$.
Associated to such a $\gamma_{\overline{\lambda}}$ we have the $\sW^k$ Verma module
$\mathbf{M}(\gamma_{\overline{\lambda}})$, and its unique irreducible quotient $\mathbf{L}(\gamma_{\overline{\lambda}})$. Both of these objects belong to $\mathcal{O}(\sW^k)$ \cite[p.\ 286]{Ara-repW1}.

There is a BRST reduction functor 
\begin{align*}
	H^0_-:  \dot{\mathcal{O}}_k\rightarrow \mathcal{O}(\sW^k).
\end{align*} 
The following sequence of results due to Arakawa is important.

\begin{thm} 
\label{thm:H0-}	
The functor $H^0_-:  \dot{\mathcal{O}}_k\rightarrow \mathcal{O}(\sW^k)$ satisfies the following properties.
\begin{enumerate}
	\item \cite[Cor.\ 7.6.2]{Ara-repW1}
	This functor is exact.
	\item Due to \cite[Thm.\ 7.5.1]{Ara-repW1} and \cite[Thm.\ 5.7]{Ara-vanishing} we have that for $\lambda\in \la{h}^*_k$, 
	\begin{align*}
		H^0_-(\mathbb{M}(\lambda))\cong\mathbf{M}(\gamma_{-w_0(\overline{\lambda})}).
	\end{align*}
	\item For $\lambda\in Pr^k_{nondeg}$, we have \cite[Cor.\ 7.6.4]{Ara-repW1}:
	\begin{align*}
		H^0_-(\mathbb{L}(\lambda))\cong \mathbf{L}(\gamma_{-w_0(\overline{\lambda})})\neq 0.
	\end{align*} 
	The last inequality is important and follows from \cite[Prop. 7.2.3]{Ara-repW1}.
	\item For a non-degenerate principal admissible number $k$, and 
	\begin{align*}
		\lambda = \Lambda( (p-h^\vee)\Lambda_0, (p'-h)\Lambda_0)\in Pr^{k}_{nondeg}
	\end{align*}
	we have (\cite[Rem.\ 7.6.5]{Ara-repW1}):
	\begin{align*}
		\sW^k&\cong H^0_-(\mathbb{M}(\lambda))\cong\mathbf{M}(\gamma_{-w_0(\overline{\lambda})})\\
		\sW_k&\cong H^0_-(\mathbb{L}(\lambda))\cong\mathbf{L}(\gamma_{-w_0(\overline{\lambda})}).
	\end{align*}
\end{enumerate}	
\end{thm}

Now we recall the resolutions of admissible weight modules $\mathbb{L}(\lambda)$ for $\la{g}$.

\begin{thm} 
\label{thm:BGG}	
Let $k$ be an admissible number and let $\lambda\in Pr^k$.
Then, there exists an exact sequence of $\la{g}$ modules of the form:
\begin{align}
	\cdots 
	\xrightarrow{d_3} 
	\bigoplus_{\substack{w\in W_\lambda \\ \len_{\lambda}(w)=2}} \mathbb{M}(w\circ \lambda)
	\xrightarrow{d_2} 
	\bigoplus_{\substack{w\in W_\lambda \\ \len_{\lambda}(w)=1}} \mathbb{M}(w\circ \lambda)
	\xrightarrow{d_1}
	\mathbb{M}(\lambda)
	\xrightarrow{\pi} 
	\mathbb{L}(\lambda)
	\rightarrow 0.
	\label{eqn:admBGG}
\end{align}
Moreover, each morphism $d_i$ ($i\geq 1$) restricted to any summand $\mathbb{M}(w\circ \lambda)$ with $\len_\lambda(w)=i$ is injective.
\end{thm}
\begin{proof}
	The resolution \eqref{eqn:admBGG} is constructed by Arakawa in \cite[Thm.\ 6.4, 6.5]{Ara-BGG}.
	Arakawa's proof is as follows. 
	By $\mathcal{O}^{\la{g}}_{[\lambda]}$ let us denote the block of category $\mathcal{O}^{\la{g}}$ for $\la{g}$ that contains the module $\mathbb{L}(\lambda)$.
	Using Fiebig's results \cite{Fei-combO}, there is an equivalence of categories 
	\begin{align}
		\mathcal{O}^{\la{g}}_{[\lambda]}\cong\mathcal{O}^{\la{g}'}_{[\lambda']}, \label{eqn:catOeqv}
	\end{align}
	where $\la{g}'$ is some symmetrizable Kac--Moody Lie algebra  whose Weyl group $W'$ is isomorphic to $W_\lambda$, $\lambda'$ is a dominant integral weight for $\la{g}'$ and $\mathcal{O}^{\la{g}'}_{[\lambda']}$ the block of $\mathcal{O}^{\la{g}'}$ containing the integrable highest-weight module $\mathbb{L}_{\la{g}'}(\lambda')$. Under this equivalence, $\mathbb{M}_{\la{g}}(w\circ \lambda)$ and $\mathbb{L}_{\la{g}}(w\circ \lambda)$ ($w\in W_\lambda)$ map to 
	$\mathbb{M}_{\la{g}'}(\phi(w)\circ \lambda')$ and $\mathbb{L}_{\la{g}'}(\phi(w)\circ \lambda')$ where $\phi:W_\lambda\xrightarrow{\cong}W'$ is an isomorphism.
	
	Now, the BGG resolution for $\mathbb{L}_{\la{g}'}(\lambda')$ given in \cite{GarLep} and \cite{RocWal-proj}
	\begin{align}
		\cdots 
		\xrightarrow{d_3'} 
		\bigoplus_{\substack{w\in W' \\ \len'(w)=2}} \mathbb{M}_{\la{g}'}(w\circ \lambda')
		\xrightarrow{d_2'} 
		\bigoplus_{\substack{w\in W' \\ \len'(w)=1}} \mathbb{M}_{\la{g}'}(w\circ \lambda')
		\xrightarrow{d_1'}
		\mathbb{M}_{\la{g}'}(\lambda')
		\xrightarrow{\pi'} 
		\mathbb{L}_{\la{g}'}(\lambda')
		\rightarrow 0
		\label{eqn:admBGG'}
	\end{align}
	is transported to the required resolution \eqref{eqn:admBGG} via the equivalence of categories \eqref{eqn:catOeqv}.
	Note that an equivalence of abelian categories is exact -- it is both  left and right adjoint \cite{Mac-cat}, and thus both right and left exact.
	
	Now, we prove the second statement. In the resolution \eqref{eqn:admBGG'}, 
	it is known by \cite[Lem.\ 8.3]{RocWal-proj} that each $d_i'$ ($i\geq 1$) restricted to a summand $\mathbb{M}_{\la{g}'}(w\circ\lambda')$ (with $w\in W'$ of length $i$) is an injection. 
	The equivalence of abelian categories \eqref{eqn:catOeqv} is exact, and thus sends injections to injections.
\end{proof}	

Assuming $k$ to be a non-degenerate principal admissible level, $\lambda\in Pr^k_{nondeg}$ and applying the exact functor $H^0_-$ to \eqref{eqn:admBGG}, we get an exact sequence of $\sW^k$ modules.
\begin{thm}
	\label{thm:WBGG}
	Let $k$ be a non-degenerate principal admissible number, and $\lambda\in Pr^k_{nondeg}$. Then, we have an exact sequence of $\sW^k$ modules:
	\begin{align}
		\cdots 
		\xrightarrow{\boldsymbol{d}_2}
		\bigoplus_{\substack{w\in W_\lambda \\ \len_{\lambda}(w)=1}} \mathbf{M}(\gamma_{-w_0(\overline{w\circ \lambda})})
		\xrightarrow{\boldsymbol{d}_1}
		\mathbf{M}(\gamma_{-w_0(\overline{\lambda})})
		\xrightarrow{\boldsymbol{\pi}}
		\mathbf{L}(\gamma_{-w_0(\overline{\lambda})})
		\rightarrow 0.
		\label{eqn:WBGG}
	\end{align}
	Moreover, each $\boldsymbol{d}_i$ ($i\geq 1$) restricted to any summand $\mathbf{M}(\gamma_{-w_0(\overline{w\circ \lambda})})$ with $\len_\lambda(w)=i$ is an injection.
\end{thm}
\begin{proof}
	The resolution \eqref{eqn:WBGG} is obtained immediately by applying $H^0_-$ to \eqref{eqn:admBGG} and using its properties from Theorem \ref{thm:H0-}. 
	From the second statement in Theorem \ref{thm:BGG},
	we have the exact sequence (for $w\in W_\lambda$ of length $i$):
	\begin{align*}
		0\rightarrow \mathbb{M}(w\circ \lambda) \xrightarrow{d_i\vert_{\mathbb{M}(w\circ \lambda)}} \bigoplus_{\substack{w\in W_\lambda\\ \len_\lambda(w)=i-1}}
		\mathbb{M}(w\circ \lambda).
	\end{align*}
	Applying the exact functor $H^0_-$ gets us the exact sequence
	\begin{align*}
		0\rightarrow \mathbf{M}(\gamma_{-w_0(\overline{w\circ \lambda})}) \xrightarrow{\boldsymbol{d}_i\vert_{\mathbf{M}(\gamma_{-w_0(\overline{w\circ \lambda})})}} \bigoplus_{\substack{w\in W_\lambda\\ \len_\lambda(w)=i-1}}
		\mathbf{M}(\gamma_{-w_0(\overline{w\circ \lambda})}).
	\end{align*}
	This proves the second part of the theorem.
\end{proof}

At non-critical levels, the Verma modules $\mathbf{M}$ and their irreducible quotients $\mathbf{L}$ are graded by conformal weights. 

\begin{thm} \cite[Prop.\ 5.6.6]{Ara-repW1}
	\label{thm:WVermachar}
	For $\lambda\in\la{h}^*_k$ with $k\neq -h^\vee$, we have:
	\begin{align}
		\ch \mathbf{M}(\gamma_{\overline{\lambda}}) = 
		\dfrac{1}{\eta(q)^\ell}
		q^{\frac{\|\overline{\lambda}+\overline{\rho}\|^2}{2(k+h^\vee)}}
		\label{eqn:WVermachar}
	\end{align}
	where $\eta$ is the Dedekind function $\eta(q) = q^{1/24}(q)_\infty$.
	In particular, the lowest conformal weight space is $1$ dimensional -- spanned by the highest weight vector.
\end{thm}

Thus, calculating the characters in \eqref{eqn:WBGG} we have the following.
\begin{thm}
	Let $k$ be a non-degenerate principal admissible number, and $\lambda\in Pr^k_{nondeg}$.
	\begin{enumerate}
	\item We have:
		\begin{align}
			\eta(q)^\ell\cdot \ch \mathbf{L}(\gamma_{-w_0(\overline{\lambda})}) 
			&=
			\sum_{\substack{w\in W_\lambda}}
			(-1)^{\len_\lambda(w)}
			q^{\frac{\|-w_0\overline{(w\circ\lambda)}+\overline{\rho}\|^2}{2(k+h^\vee)}}
			\label{eqn:wchar}.
		\end{align}
		\item The map 
		\begin{align}
			w\mapsto \frac{\|-w_0\overline{(w\circ\lambda)}+\overline{\rho}\|^2}{2(k+h^\vee)}
			\label{eqn:mapconfwt}
		\end{align}
		defined on $W_\lambda$ achieves its (global) minimum uniquely at $w=1$.
	\end{enumerate}
\end{thm}	
\begin{proof}
	The first assertion is simply the Euler--Poincar\'e principle applied to \eqref{eqn:WBGG} using \eqref{eqn:WVermachar}.
	
	It is the second assertion which is very important for us.
	
	For the purpose of this proof, let us denote $\mathbf{M}(\gamma_{-w_0(\overline{w\circ \lambda})})$ in a shorthand as $\mathbf{M}(\lambda, w)$, its highest weight vector as $v(\lambda, w)$ which has conformal weight $\Delta(\lambda, w)$. With this notation, the map \eqref{eqn:mapconfwt} is simply $w\mapsto \Delta(\lambda,w)+\frac{\ell}{24}$. We may and do ignore the term $\frac{\ell}{24}$.

	We must thus prove that for all $w\in W_\lambda$, $w\neq 1$ we have $\Delta(\lambda,w) > \Delta(\lambda,1)$. We proceed by induction on $\len_\lambda(w)$.
	
	First let $\len_\lambda(w)=1$. Now, by Theorem \ref{thm:WBGG}, $\boldsymbol{d}_1\vert_{\mathbf{M}(\lambda,w)}$ is an injection, thus $0\neq \boldsymbol{d}_1(v(\lambda,w))\in \mathbf{M}(\lambda,1)$. All maps in \eqref{eqn:WBGG} (in particular, $\boldsymbol{d}_1$) preserve conformal weights. Thus, $\Delta(\lambda,w)\geq \Delta(\lambda,1)$. If the equality were to hold, then, since the lowest conformal weight space of $\mathbf{M}(v,1)$ is one-dimensional (Theorem \ref{thm:WVermachar}), we would have $0\neq \boldsymbol{d}_1(v(\lambda,w))$ to be proportional to $v(\lambda,1)$.
	In particular, $v(\lambda,1)\in \im(\boldsymbol{d}_1) = \ker(\boldsymbol{\pi})$ (by exactness of \eqref{eqn:WBGG}). This would imply that $\mathbf{L}(-\gamma_{w_0(\overline \lambda)})=0$ contradicting Theorem \ref{thm:H0-} part (3).
	Thus we have proved that for all $w\in W_\lambda$ with $\len_\lambda(w)=1$, $\Delta(\lambda,w)>\Delta(\lambda,1)$.
	
	Now let $\len_\lambda(w)=i$, $i>1$. Again, $\boldsymbol{d}_i(v(\lambda,w))\neq 0$ by Theorem \ref{thm:WBGG}. Looking at conformal weights, we conclude 
	\begin{align*}
		\Delta(\lambda,w)\geq \min_{\substack{w\in W_\lambda,\\ \len(w)=i-1}}\Delta(\lambda,w) > \Delta(\lambda,1), 
	\end{align*}
	where the last inequality is by the induction hypothesis.
	This finishes the proof.
\end{proof}

We rewrite the characters more explicitly as follows.
\begin{prop}
	\label{cor:Wminexp}
	Let $k=\frac{p}{p'}-h^\vee$ be a non-degenerate principal admissible number, let $j\in J$ and finally, let $\lambda=\Lambda(\sigma_j\nu,\mu)$ for $(\nu,\mu)\in P^{p-h^\vee}_+\times P^{\vee, p'-h}_+$.
	\begin{enumerate}
		\item We have:
		\begin{align}
			&{\eta(q)^\ell}\, \ch \mathbf{L}(\gamma_{-w_0(\overline{\lambda})})\notag\\
			&=
			(-1)^{2 (\overline\Lambda_j,\overline\rho)}
			\sum_{\substack{\alpha\in \overline{Q}^\vee - \overline{\Lambda}_i, \overline{w}\in \overline{W}}}
			(-1)^{\len(\overline{w})}
			q^{\frac{pp'}{2}
				\left\| \alpha +\frac{\overline{\nu}+\overline{\rho}}{p}-\overline{w}\frac{\overline{\mu}+\overline{\rho^\vee}}{p'}\right\|^2}
			\label{eqn:wchar3}
		\end{align}
		where the index $i\in J$ is such that $\overline{\sigma_j}^{-1}\overline{\Lambda}_j=-\overline{\Lambda}_i$, equivalently, $\sigma_i=\sigma_j^{-1}$; see \eqref{eqn:Jij}.
		\item The map
		\begin{align}
			(\alpha, \overline{w}) \mapsto \frac{pp'}{2}
			\left\| \alpha +\frac{\overline{\nu}+\overline{\rho}}{p}-\overline{w}\frac{\overline{\mu}+\overline{\rho^\vee}}{p'}\right\|^2
			\label{eqn:wchar3expfun}
		\end{align}
		defined on $(\overline{Q}^\vee-\overline{\Lambda}_i)\times \overline{W}$ achieves its (global) minimum uniquely at $\alpha= -\overline{\Lambda}_i$, $\overline{w}=\overline{\sigma}_j^{-1}$.
	\end{enumerate}
\end{prop}
\begin{proof}
	We rewrite each term in \eqref{eqn:wchar} using the description of $W_{\Lambda(\sigma_j\nu,\mu)}$ from \eqref{eqn:Wnumu}. 
	Let
	\begin{align*}
		w= t_{-\overline{\mu+\rho^\vee}}\,\overline{w}t_{p'\alpha}\,t_{\overline{\mu+\rho^\vee}}
		\in W_{\Lambda(\sigma_j\nu,\mu)},
	\end{align*}
	with $\overline{w}\in\overline{W}, \alpha\in \overline{Q}^\vee$.
	Note that 
	\begin{align}
		(-1)^{\len_\lambda(w)}=
		(-1)^{\len(\overline{w}t_{p'\alpha})}
		=(-1)^{\len(\overline{w})}=(-1)^{\len(\overline{w}^{-1})}.
	\end{align}
	Also, notice that $w=1$ corresponds to $\overline{w}=1,\alpha=0$.
	Then, simplifying the exponents of terms in \eqref{eqn:wchar}
	\begin{align*}
		&\|-w_0\overline{(w\circ\Lambda(\sigma_j\nu,\mu))}+\overline{\rho}\|^2
		=\|-w_0\overline{(w\circ\Lambda(\sigma_j\nu,\mu))}-w_0\overline{\rho}\|^2\notag\\
		&=\|\overline{(w\circ\Lambda(\sigma_j\nu,\mu))}+\overline{\rho}\|^2
		=\|\overline{w(\Lambda(\sigma_j\nu,\mu)+\rho) -\rho}+\overline{\rho}\|^2=\|\overline{w(\Lambda(\sigma_j\nu,\mu)+\rho)}\|^2\notag\\
		&=\left\|\overline{w\left(\sigma_j(\nu+\rho) - \frac{p}{p'}(\mu+\rho^\vee) +\frac{p}{p'}\Lambda_0+\cdots \delta \right) }\right\|^2\notag\\
		&=\left\|\overline{t_{-\overline{\mu+\rho^\vee}}\overline{w}\left(\sigma_j(\nu+\rho) - \frac{p}{p'}(\mu+\rho^\vee) +\frac{p}{p'}\Lambda_0  +\frac{p}{p'}(p'\alpha+\overline{\mu+\rho^\vee})+\cdots\delta\right)}\right\|^2\notag\\	
		&=\left\|\overline{t_{-\overline{\mu+\rho^\vee}}\overline{w}\left(\sigma_j(\nu+\rho) - \frac{p}{p'}(\mu+\rho^\vee-\overline{\mu+\rho^\vee}) +\frac{p}{p'}\Lambda_0 
			+p\alpha +\cdots\delta\right)}\right\|^2\notag\\	
		&=\left\|\overline{t_{-\overline{\mu+\rho^\vee}}\left(\overline{w}\sigma_j(\nu+\rho) - \frac{p}{p'}(\mu+\rho^\vee-\overline{\mu+\rho^\vee}) +\frac{p}{p'}\Lambda_0 
			+p\overline{w}\alpha +\cdots\delta\right)}\right\|^2\notag\\	
		&=\left\|\overline{\overline{w}\sigma_j(\nu+\rho) - \frac{p}{p'}(\mu+\rho^\vee) +\frac{p}{p'}\Lambda_0
			+p\overline{w}\alpha +\cdots\delta}\right\|^2\notag\\	
		&=\left\|p\overline{w}\alpha+\overline{w}(\overline{\sigma_j(\nu+\rho)}) - \frac{p}{p'}(\overline{\mu+\rho^\vee}) 
		\right\|^2\notag\\
		&=\left\|p\alpha+\overline{\sigma_j(\nu+\rho)} - \frac{p}{p'}\overline{w}^{-1}(\overline{\mu+\rho^\vee}) 
		\right\|^2\notag\\
		&=\left\|p(\alpha+\overline{\Lambda}_j)+\overline{\sigma}_j\overline{(\nu+\rho)} - \frac{p}{p'}\overline{w}^{-1}(\overline{\mu+\rho^\vee}) 
		\right\|^2\notag\\
		&=\left\|p\overline{\sigma}_j^{-1}(\alpha+\overline{\Lambda}_j)+\overline{(\nu+\rho)} - \frac{p}{p'}\overline{\sigma}_j^{-1}\overline{w}^{-1}(\overline{\mu+\rho^\vee}) 
		\right\|^2.
		\end{align*}
	Adjusting for the factor $2(k+h^\vee) = 2p/p'$ and by using properties of $(-1)^{\len}$ we get
	\begin{align*}
	&{\eta(q)^\ell}\, \ch \mathbf{L}(\gamma_{-w_0(\overline{\lambda})})\notag\\
	&=
	(-1)^{\len(\overline{\sigma}_j)}
	\sum_{\alpha\in \overline{Q}^\vee, \overline{w}\in\overline{W}}
	(-1)^{\len(\overline{\sigma}_j^{-1}\overline{w}^{-1})}
	q^{\frac{pp'}{2}\left\|\overline{\sigma}_j^{-1}(\alpha+\overline{\Lambda}_j)+
	\frac{\overline{\nu+\rho}}{p} - \frac{\overline{\sigma}_j^{-1}\overline{w}^{-1}(\overline{\mu+\rho^\vee})}{p'} 
		\right\|^2},
	\end{align*}
	with $\alpha=0$ and $\overline{w}=1$ providing the unique minimizer to the function 
	\begin{align*}
		(\alpha , \overline{w})\mapsto 
		\frac{pp'}{2}\left\|\overline{\sigma}_j^{-1}(\alpha+\overline{\Lambda}_j)+
		\frac{\overline{\nu+\rho}}{p} - \frac{\overline{\sigma}_j^{-1}\overline{w}^{-1}(\overline{\mu+\rho^\vee})}{p'} 
		\right\|^2
	\end{align*} 
	defined on $\overline{Q}^\vee\times \overline{W}$.
	
	Now, we reindex the summation, noting that $\overline{\sigma}_j(\overline{Q}^\vee +\overline{\Lambda}_j)
	=\overline{Q}^\vee -\overline{\Lambda}_i$.
	
	We thus arrive at the intended expression \eqref{eqn:wchar3}, with $\alpha=-\overline{\Lambda}_i$ and $\overline{w}=\overline{\sigma}_j^{-1}$ providing the unique  minimizer to \eqref{eqn:wchar3expfun}. Note that we have used 
	\eqref{eqn:sgnsigmaj} to replace $(-1)^{\len(\overline\sigma_j)}$.
\end{proof}

\begin{rem}
	In \cite{Kan-torus}, we had considered the ``shifted'' characters. We can now see them as characters of certain $\sW$ algebra modules (up to sign, pure powers of $q$ and $\eta(q)$ factors). In particular, these characters exhibit modular invariance.
\end{rem}

\section{Coloured invariants of torus knots}

Throughout this section, fix a finite-dimensional simple Lie algebra $\overline{\la{g}}$ over $\CC$.

We shall need additional notation beyond what is recalled in the previous section. We will continue to follow the convention that quantities related to $\overline{\la{g}}$ will be denoted by an overline.

Recall that the set of positive roots for $\overline{\la{g}}$ is denoted by $\overline{\Delta}_+$, and we will denote the containment $\alpha\in\overline \Delta_+$ alternately as $\alpha\succ 0$. The set of dominant integral weights by $\overline{P}_+$,  the set of strongly dominant integral weights by $\overline{P}_+^\circ=\overline{\rho}+\overline{P}_+$. 
The irreducible highest-weight modules by $L(\lambda)$ for $\lambda\in \overline{P}_+$, the corresponding weight multiplicities by $m_{\lambda}(\mu)$ for $\mu\in \overline{P}$.

For a formal variable $q$, define a linear map:
\begin{align*}
	\qdim:\CC[\overline{P}]&\rightarrow \CC[\overline{P}] \\
	e^{\mu} &\mapsto q^{(\mu,\overline{\rho})}.
\end{align*}

Denoting the set of weights of $L(\lambda)$ by $\wt(\lambda)$, the formal character of $L(\lambda)$ is:
\begin{align}
	\ch(L(\lambda))
	=\sum_{\mu\in \wt(\lambda)}m_{\lambda}(\mu)e^{\mu}
	=\dfrac{\sum_{\overline w\in \overline W}(-1)^{\len(\overline w)} e^{\overline w(\lambda+\overline{\rho})}}
	{\sum_{\overline w\in \overline W}(-1)^{\len(\overline w)} e^{\overline w\,\overline{\rho}}}
	\in \CC[\overline{P}].
	\label{eqn:char}
\end{align}

Recall that  for all $\lambda\in \overline P_+$ (see the proof of \cite[Cor.\ 24.6]{FulHar} with $e^{t}=q$)
\begin{align*}
	\qdim\left( \sum_{\overline w\in \overline W}(-1)^{\len(\overline w)} e^{\overline w(\lambda+\overline \rho)}\right) = \prod_{\alpha\succ 0} 
	\left( 
	q^{\frac{1}{2}(\lambda+\overline \rho,\alpha)}-
	q^{-\frac{1}{2}(\lambda+\overline \rho,\alpha)}
	\right).
\end{align*}		
Denoting the denominator of \eqref{eqn:char} by $\sD$, we therefore have:
\begin{align}
	\qdim(\sD)&=\sum_{\overline w\in \overline W}(-1)^{\len(\overline w)} q^{(\overline w\, \overline \rho,\overline \rho)}
	=\prod_{\alpha\succ 0} (q^{\frac{1}{2}(\overline \rho,\alpha)}-
	q^{-\frac{1}{2}(\overline \rho,\alpha)})\notag\\
	&=(-1)^{\len(w_0)}q^{-\|\overline \rho \|^2}\prod_{\alpha\succ 0} 
	(1-q^{(\overline \rho,\alpha)}),
	\label{eqn:qdim}
\end{align}
where we have used that $|\overline \Delta_+| = \len(w_0)$.
Further,
\begin{align*}
	\qdim(\ch(L(\lambda))) = 
	\prod_{\alpha\succ 0} 
	\dfrac{
		q^{\frac{1}{2}(\lambda+\overline \rho,\alpha)}-
		q^{-\frac{1}{2}(\lambda+\overline \rho,\alpha)}
	}
	{
		q^{\frac{1}{2}(\overline \rho,\alpha)}-
		q^{-\frac{1}{2}(\overline \rho,\alpha)}
	}.
\end{align*}		

We define the $p$th Adams operation and the corresponding plethysm multiplicities $m_{\lambda,p}^{\mu}$ ($\lambda,\mu\in \overline P_+$, $p\geq 1$) by:
\begin{align}
	\psi_p:\CC[\overline P]&\rightarrow \CC[\overline P] \nonumber\\
	e^{\mu} &\mapsto e^{p\mu},		
	\label{eqn:adams}
	\\
	\psi_p(\ch(L(\lambda)))&=\sum_{\mu\in \overline P_+}m_{\lambda,p}^{\mu}\ch(L(\mu)).
	\label{eqn:plethysm}
\end{align}

Finally, for $\lambda\in \overline P_+$, we define the ribbon twist by 
\begin{align*}
	\theta_{\lambda} = q^{\frac{1}{2}(\lambda,\lambda+2\overline \rho)}.
\end{align*}	

\subsection{The Rosso--Jones formula}

Let $p,p'\ge 1$ be a pair of coprime integers, and let $T(p,p')$ be the corresponding torus knot with writhe $pp'$ (see \cite{Mor-coloured}).
Fix $\lambda\in\overline P_+$. The (un-normalized, framing dependent) invariant for $T(p,p')$ coloured by $L(\lambda)$ is \cite{RosJon-torus}, \cite{Mor-coloured}, \cite{GarVuo}:
\begin{align}
	\jones_{T(p,p')}(\lambda) = \sum_{\mu\in\overline  P_+}m_{\lambda,p}^\mu
	\cdot \theta_{\mu}^{p'/p}\cdot \qdim(\ch(L(\mu))).
	\label{eqn:RossoJones}
\end{align}
Note that the corresponding formula \cite[Eq.\ (5)]{GarVuo} involves the extra factor $$\frac{\theta_\lambda^{-pp'}}{\qdim(\ch(L(\lambda)))}$$ as well. However,  
we ignore this factor, in effect fixing an appropriate framing and ``un''-normalizing the invariant.
Straight-forwardly generalizing Morton's argument \cite{Mor-coloured}, we can put \eqref{eqn:RossoJones} in a form more amenable to be compared with principal $\sW$ algebra characters. In the $\overline{\la{g}}=\la{sl}_r$ case with $\lambda=n\overline \Lambda_1$, this was explained in \cite{Kan-torus}. The argument is exactly the same in general, and we provide it here for completeness, closely following \cite{Kan-torus} (which, in turn, followed Morton's work \cite{Mor-coloured}).

\begin{thm}
	For $\lambda\in \overline{P}_+$, and coprime $p,p'\geq 1$, the formula for the coloured invariant can be written as:
	\begin{align}
		&\jones_{T(p,p')}(\lambda)
		=\dfrac{q^{-\frac{pp'}{2}\left(\frac{1}{p}-\frac{1}{p'}\right)^2\|\overline\rho\|^2} }{\prod_{\alpha\succ 0 }(1-q^{(\alpha,\overline\rho)})}
		\sum_{\substack{\mu\in \wt(\lambda)\\ \overline w\in \overline W}}m_{\lambda}(\mu)\cdot (-1)^{\len(\overline w)} \cdot q^{\frac{pp'}{2}\left\| \mu +\frac{\overline \rho}{p}-\overline{w}\frac{\overline\rho}{p'}\right\|^2}.
		\label{eqn:jones}
	\end{align}
\end{thm}
\begin{proof}
	Firstly, using \eqref{eqn:char}, \eqref{eqn:adams} and \eqref{eqn:plethysm}, we may write, for $\lambda\in\overline P_+$,
	\begin{align}
		&\left(\sum_{\overline w \in \overline W }(-1)^{\len(\overline w)}e^{\overline w\,\overline\rho} \right)\psi_p(\ch(L(\lambda)))=
		\sum_{\substack{\mu\in \wt(\lambda)\\\overline w\in \overline W}}(-1)^{\len(\overline w)}m_{\lambda}(\mu) e^{p\mu+\overline w\,\overline \rho}\notag\\
		&\quad=\sum_{\substack{\mu\in \overline P_+\\ \overline w\in \overline W}}(-1)^{\len(\overline w)} m_{\lambda,p}^{\mu}e^{\overline w(\mu+\overline \rho)}
		=\sum_{\substack{\mu'\in \overline P_+^\circ\\ \overline w \in \overline W}}(-1)^{\len(\overline w)} m_{\lambda,p}^{\mu'-\overline \rho}e^{\overline w\mu'}.
		\label{eqn:plethysm1}
	\end{align}
	Now, any $\mu''\in \overline W\,\overline P_+^\circ$ can be uniquely written as $\mu''=\overline w\mu'$ with $\overline w \in \overline W$ and $\mu\in\overline  P_+^\circ$ (which follows from \cite[Lem.\ 13.2A]{Hum-la}), and therefore we define, for $\mu'' \in \overline W\, \overline P_+^\circ$:
	\begin{align*}
		M^{\mu''}_{\lambda,p} = (-1)^{\len(\overline w)} m_{\lambda,p}^{\mu'-\overline \rho}.
	\end{align*}
	Further, for $\mu''\in \overline P\backslash \overline W\,\overline P_+^\circ$, we define
	\begin{align*}
		M^{\mu''}_{\lambda,p} = 0,
	\end{align*}
	since the coefficient of such $e^{\mu''}$ is $0$ on the RHS of \eqref{eqn:plethysm1} (and thus, in each expression of \eqref{eqn:plethysm1}).
	With this, we may write \eqref{eqn:plethysm1} as:
	\begin{align*}
		\sum_{\substack{\mu\in \wt(\lambda)\\\overline w \in \overline W}}(-1)^{\len(\overline w)}m_{\lambda}(\mu) e^{p\mu+\overline w\,\overline \rho}
		&
		=\sum_{\substack{\mu''\in \overline W\,\overline P_+^\circ}}M^{\mu''}_{\lambda,p}e^{\mu''}
		=\sum_{\substack{\mu''\in \overline P}}M^{\mu''}_{\lambda,p}e^{\mu''}
	\end{align*}
	equivalently for $\mu''\in \overline P$, $\lambda\in\overline  P_+$,
	\begin{align*}
		M^{\mu''}_{\lambda,p} = \coeff_{e^{\mu''}}\left(\sum_{\substack{\mu\in \wt(\lambda)\\\overline w \in \overline W}}(-1)^{\len(\overline w)}m_{\lambda}(\mu) e^{p\mu+\overline w\,\overline \rho}\right). 
	\end{align*}
	Now we see:
	\begin{align*}
		\jones&{}_{T(p,p')}(\lambda) = \sum_{\mu\in \overline P_+}m_{\lambda,p}^\mu
		\cdot \theta_{\mu}^{p'/p}\cdot \qdim(\ch(L(\mu)))
		\nonumber\\
		&=
		\dfrac{1}{\qdim(\sD)}\sum_{\substack{\mu\in \overline P_+,\overline w \in \overline W}}(-1)^{\len(\overline w)} \cdot m_{\lambda,p}^\mu
		\cdot q^{\frac{p'}{2p}(\mu,\mu+2\overline \rho)}\cdot  q^{(\overline w(\mu+\overline \rho),\overline \rho)}
		\nonumber\\
		&=
		\dfrac{1}{\qdim(\sD)}\sum_{\substack{\mu'\in \overline P_+^\circ,\overline w \in \overline W}}(-1)^{\len(\overline w)} \cdot m_{\lambda,p}^{\mu'-\overline \rho}
		\cdot q^{\frac{p'}{2p}(\mu'-\overline \rho,\mu'+\overline \rho)}\cdot  q^{(\overline w\mu',\overline \rho)}
		\nonumber\\
		&=
		\dfrac{q^{-\frac{p'}{2p}\|\overline \rho\|^2}}{\qdim(\sD)}\sum_{\substack{\mu'\in \overline P_+^\circ,\overline w \in \overline W}}(-1)^{\len(\overline w)} \cdot m_{\lambda,p}^{\mu'-\overline \rho}
		\cdot q^{\frac{p'}{2p}\|\overline w\mu'\|^2}\cdot  q^{(\overline w\mu',\overline \rho)}\\
		&=
		\dfrac{q^{-\frac{p'}{2p}\|\overline \rho\|^2}}{\qdim(\sD)}\sum_{\substack{\mu''\in \overline W\,\overline P_+^\circ}}M^{\mu''}_{\lambda,p}
		\cdot q^{\frac{p'}{2p}\| \mu''\|^2}\cdot  q^{(\mu'',\overline \rho)}
		\nonumber\\
		&=
		\dfrac{q^{-\frac{p'}{2p}\|\overline \rho\|^2}}{\qdim(\sD)}\sum_{\substack{\mu''\in \overline P}}M^{\mu''}_{\lambda,p}
		\cdot q^{\frac{p'}{2p}\|\mu''\|^2}\cdot  q^{(\mu'',\overline \rho)}
		\nonumber\\
		&=
		\dfrac{q^{-\frac{p'}{2p}\|\overline \rho\|^2}}{\qdim(\sD)}\sum_{\substack{\mu''\in \overline P}}\coeff_{e^{\mu''}}\left(\sum_{\substack{\mu\in \wt(\lambda)\\\overline w \in \overline W}}(-1)^{\len(\overline w)}m_{\lambda}(\mu) e^{p\mu+\overline w\,\overline \rho}\right)
		\cdot q^{\frac{p'}{2p}\|\mu''\|^2+(\mu'',\overline \rho)}
		\nonumber\\
		&=\dfrac{q^{-\frac{p'}{2p}\|\overline \rho\|^2}}{\qdim(\sD)}
		\sum_{\substack{\mu\in \wt(\lambda)\\\overline w \in \overline W}}(-1)^{\len(\overline w)}m_{\lambda}(\mu) 
		\cdot q^{\frac{p'}{2p}\|p\mu+\overline w\,\overline \rho\|^2+(p\mu+\overline w\,\overline \rho,\overline \rho)}.
	\end{align*}
	Now we note that $\overline{w} \wt(\lambda)=\wt(\lambda)$ and $m_{\lambda}(\mu)=m_{\lambda}(\overline w\mu)$ for all $\overline w \in \overline W$. Reindexing the sum,  and using the invariance of $\|\cdot\|^2$ with respect to $\overline W$, we get:
	\begin{align*}
		\dfrac{q^{-\frac{p'}{2p}\|\overline \rho\|^2}}{\qdim(\sD)}
		\sum_{\substack{\mu\in \wt(\lambda)\\\overline w \in \overline W}}(-1)^{\len(\overline w)}m_{\lambda}(\mu) 
		\cdot q^{\frac{p'}{2p}\|p\mu+\overline \rho\|^2+(p\mu+\overline \rho,\overline{w}^{-1}\overline \rho)}
	\end{align*}
	Reindexing the sum over $\overline{W}$ by writing $\overline{w}^{-1}\mapsto \overline{w}w_0$, using $w_0\overline\rho=-\overline\rho$, and the properties of $\len$, this expression becomes:
	\begin{align*}
		\dfrac{(-1)^{\len(w_0)}q^{-\frac{p'}{2p}\|\overline \rho\|^2}}{\qdim(\sD)}
		\sum_{\substack{\mu\in \wt(\lambda)\\\overline w \in \overline W}}(-1)^{\len(\overline w)}m_{\lambda}(\mu) 
		\cdot q^{\frac{p'}{2p}\|p\mu+\overline \rho\|^2-(p\mu+\overline \rho,\overline{w}\,\overline \rho)}.
	\end{align*}
	Using the explicit expression for $\qdim(\sD)$ from \eqref{eqn:qdim} along with a few algebraic manipulations, we now get the required result.
\end{proof}

\begin{rem}\label{rem:pp'invariance}
	The left-hand side of \eqref{eqn:jones} is invariant under $p\longleftrightarrow p'$ since $T(p,p')=T(p',p)$.
	It is an easy exercise to see that the right-hand side of \eqref{eqn:jones} is also invariant under $p\longleftrightarrow p'$.
\end{rem}

\subsection{Minimum exponent}

We now renormalize the Jones polynomial by dividing through by its trailing monomial.
We shall denote this as:
\begin{align*}
	\widehat{\jones} := \frac{\jones}{{\mathrm{trailing\,\,monomial\,\,of\,\,} \jones}}
\end{align*}

To understand $\widehat{\jones}$ we now find conditions on $p,p'$ such that the minimum exponent in \eqref{eqn:jones} occurs at a unique location $(\mu,\overline w)\in \wt(\lambda)\times\overline W$.

We record a useful lemma.
\begin{lem}\label{lem:nonzeromult}
Let $\overline{\la{g}}$ be a any finite-dimensional simple Lie algebra.
Let $\lambda\in \overline P_+\backslash\{0\}$ and let $\mu\in \lambda+\overline Q$. Then, for all $n\gg 0$ satisfying $n\lambda\in \lambda+\overline Q$, we have that:
\begin{align*}
	m_{n\lambda}(\mu) \neq 0.
\end{align*}
\end{lem}
\begin{proof}
	(Since $\mu\in \lambda+\overline Q$, $n\lambda\in \lambda+\overline{Q}$ is a necessary condition for the multiplicity to be non-zero.)
	Let
	\begin{align*}
		\lambda&=\sum_{1\leq i\leq \ell} l_i\overline \Lambda_i,
	\end{align*}
	with $l_i\in\ZZ_{\geq 0}$, at least one of which is strictly positive (since $\lambda\neq 0$). 
	Let 
	\begin{align}
		\overline\Lambda_i = \sum_{1\leq j\leq \ell}D_{ij}\alpha_j.
		\label{eqn:wtstoroots}
	\end{align}
	Then, the matrix $D$ is the inverse transpose of the cartan matrix of $\overline{\la{g}}$. Crucially, all entries of $D$ are strictly positive (rational numbers), see for instance \cite[Table 2]{OniVin}.
	
	Since
	\begin{align*}
		\lambda&=\sum_{1\leq i,j\leq \ell} l_iD_{ij}\alpha_j=\sum_{1\leq j\leq \ell} \left(\sum_{1\leq i\leq \ell}l_iD_{ij}\right)\alpha_j,
	\end{align*}
	each $\alpha_j$ coordinate of $\lambda$ is a strictly positive rational number. Consequently, for the fixed weight $\mu$, and for all large enough integers $n\gg 0$, we must have each $\alpha_j$ coordinate of $n\lambda - \mu$ is a strictly positive rational number.
	
	Now if we restrict to those $n$ such that $n\lambda\in\lambda+ \overline Q$, then, using that $\mu\in \lambda+ \overline Q$, we must have that $n\lambda-\mu \in \overline Q$, in particular, each $\alpha_j$ coordinate of $n\lambda-\mu$ must also be integral.
	
	We have thus established that for all $n\gg 0$ satisfying $n\lambda\in \lambda+\overline Q$, each $\alpha_j$ coordinate of $n\lambda-\mu$ is a strictly positive integer. In particular, $n\lambda- \mu \in \overline{Q}_+$ where
	\begin{align*}
		\overline{Q}_+ = \ZZ_{\geq 0}\alpha_1+\cdots+\ZZ_{\geq 0}\alpha_\ell.
	\end{align*}
	
	If $\mu\in \overline{P}_+$, then this implies $m_{n\lambda}(\mu)\neq 0$, \cite[Sec.\ 13.4, 21.3]{Hum-la}.

	If $\mu\not\in \overline{P}_+$, then, choose $\overline w\in \overline W$ such that $\overline w\mu\in \overline P_+$. Since $\overline w\mu\in \mu+\overline Q=\lambda+\overline Q$, the above argument establishes that for all $n\gg 0$ with $n\lambda\in \lambda+\overline Q$, $m_{n\lambda}(\overline w\mu)\neq 0$. However, $m_{n\lambda}(\mu)=m_{n\lambda}(\overline w\mu)$, which finishes the proof.
\end{proof}

One easy but useful corollary of this result is the following.
\begin{cor}
	\label{cor:limwtset}
	Let $\overline{\la{g}}$ be any finite-dimensional Lie algebra. Let $\lambda\in \overline P_+$, $\lambda\neq 0$. Let $n_1, n_2,\cdots$ be an increasing sequence of positive integers such that $n_i\lambda\in \lambda+\overline{Q}$ for all $i$. Then, we have that:
	\begin{align*}
		\wt(n_1\lambda)\subseteq \wt(n_2\lambda)\subseteq \cdots\quad \mathrm{and}\quad\, \bigcup_{i}\wt(n_i\lambda)=\lambda+\overline{Q}.
	\end{align*}
\end{cor}
\begin{proof}
	Let $\mu\in \wt(n_i\lambda)$, which is equivalent to assuming 
	$n_i\lambda-\overline{w}\mu\in\overline{Q}_+$ for all $\overline w\in \overline W$ (\cite[Sec.\ 21.3]{Hum-la}). Given that $n_i<n_{i+1}$ and $n_i\lambda, n_{i+1}\lambda\in \lambda+\overline{Q}$, we get $n_{i+1}\lambda-n_i\lambda\in\overline{Q}\cap \overline{P}_+$. 
	However, $\overline{Q}\cap \overline{P}_+\subset \overline{Q}_+$, since every entry $D_{ij}$ in \eqref{eqn:wtstoroots} is strictly positive (and rational). Therefore, $n_{i+1}\lambda-n_i\lambda\in\overline{Q}_+$.
	Thus, $n_{i+1}\lambda- \overline w\mu\in \overline{Q}_+$ for all $\overline w\in \overline W$, and hence $\mu\in \wt(n_{i+1}\lambda)$. The second assertion is immediate from the previous lemma.
\end{proof}	

Consider the simply-laced case, i.e., $\overline{\la{g}}$ is of type $\lietype{A}$, $\lietype{D}$, or $\lietype{E}$.  Recall that in this case, several notions simplify; see Section \ref{sec:simplylaced}.
Recall the set $J$ from \eqref{eqn:Jset}. The map
$j\mapsto \overline{\Lambda}_j+\overline{Q}$ is a bijection from $J$ to $\overline{P}/\overline{Q}$ (note that $\overline{\Lambda}_0=0$).
Therefore, given $\lambda\in \overline{P}_+$, there exists a unique $j\in J$ such that:
\begin{align*}
	\lambda \in \overline \Lambda_j+\overline{Q}
\end{align*}
and thus,
\begin{align*}
	\wt(\lambda) \subseteq \overline \Lambda_j+\overline{Q} =
	-\overline \Lambda_i+\overline{Q},
\end{align*}
where $i\in J$ is as in \eqref{eqn:Jij}.
Now, as long as $-\overline{\Lambda}_i$ appears in $\wt(\lambda)$ (i.e., $m_{\lambda}(-\overline\Lambda_i)\neq 0$), we may use Proposition \ref{cor:Wminexp} to pin down the minimum exponent term in \eqref{eqn:jones}.

Using our analysis of $\sW$ algebra characters, we now arrive at the following theorem.
\begin{thm}
	\label{thm:ADEjones_hat}
	Let $\overline{\la{g}}$ be simply-laced and let $p,p'\geq h^\vee$ be coprime positive integers (where $h^\vee$ is the dual Coxeter number).
	
	Let $\lambda\in \overline P_+$ be non-zero. 
	Let $j\in J$ be such that $\lambda\in \overline{\Lambda}_j+\overline Q$. 
	
	Then, for all $n\gg 0$ with $n\lambda\in \lambda+\overline{Q}$, we have that:
	
	\begin{align}
	\widehat{\jones}_{T(p,p')}(n\lambda) &= 
	\dfrac{(-1)^{2(\overline\Lambda_j,\overline{\rho})}
		q^{-\frac{pp'}{2}\left\|\overline\Lambda_j +\overline{\sigma}_j\frac{\overline{\rho}}{p}
		-\frac{\overline{\rho}}{p'} \right\|^2} }
	{\prod_{\alpha\succ 0 }(1-q^{(\alpha,\overline\rho)})}\notag\\
	&\,\,\times
	\sum_{\substack{\mu\in \wt(n\lambda)\\ \overline w\in \overline W}}\dfrac{m_{n\lambda}(\mu)}{m_{n\lambda}(-\overline\Lambda_i)}\cdot (-1)^{\len(\overline w)} \cdot q^{\frac{pp'}{2}\left\| \mu +\frac{\overline \rho}{p}-\overline{w}\frac{\overline\rho}{p'}\right\|^2}.
	\label{eqn:jones_hat}
	\end{align}
\end{thm}
\begin{proof}
	Since $n\lambda\in\overline \Lambda_j+\overline Q$, we have $\wt(n\lambda)\subseteq \overline{\Lambda}_j+\overline Q=-\overline{\Lambda}_i+\overline{Q}$. From Lemma \ref{lem:nonzeromult} we choose $n\gg 0$ such that $-\overline\Lambda_i\in \wt(n\lambda)$ (i.e., $m_{n\lambda}(-\overline\Lambda_i)\neq 0$).
	Thus, for $\jones_{T(p,p')}(n\lambda)$, the sum in \eqref{eqn:jones} is over a subset of $(-\overline{\Lambda}_i+\overline{Q}) \times\overline W$ which includes the point $(-\overline\Lambda_i,\overline \sigma_j^{-1})$.
	
	Taking $\nu= (p-h^\vee)\Lambda_0$ and $\mu=(p'-h)\Lambda_0$ in Proposition \ref{cor:Wminexp} (the $\mu$ there should not be confused with the $\mu$ in \eqref{eqn:jones_hat}), we thus see that the 
	trailing monomial in $\jones_{T(p,p')}(n\lambda)$ is contributed uniquely by the term corresponding to $\alpha=-\overline\Lambda_i, \overline{w}=\overline \sigma_j^{-1}$.
	
	Therefore, the trailing monomial in $\jones_{T(p,p')}(n\lambda)$ in \eqref{eqn:jones} equals
	\begin{align*}
		q^{-\frac{pp'}{2}\left(\frac{1}{p}-\frac{1}{p'}\right)\|\overline\rho\|^2}
		\cdot 
		m_{n\lambda}(-\overline{\Lambda}_i)\cdot (-1)^{\len(\overline\sigma_j)}
		\cdot 
		q^
		{\frac{pp'}{2}\left \| -\overline\Lambda_i +\frac{\overline\rho}{p} - \overline \sigma_j^{-1}\frac{\overline \rho}{p'} \right\|^2}.
	\end{align*}
	Using \eqref{eqn:sgnsigmaj} and \eqref{eqn:Jij} now gets us the result.
\end{proof}

\begin{rem}
	Under the conditions of the theorem above, if $\lambda\in \overline Q \cap \overline P_+$, (equivalently, if $j=0$), then the trailing monomial of $\jones_{T(p,p')}(n\lambda)$ is $m_{n\lambda}(0)q^0$, uniquely contributed by $\alpha=-\overline\Lambda_i=0$ and $\overline w = \overline \sigma_j^{-1}=1$ from \eqref{eqn:jones}. In this case, the result holds for all $n\geq 1$ since $0\in \wt(\lambda)$. 
\end{rem}
\begin{rem}
	Recall that this is an un-normalized, framing-dependent invariant, and this makes the degree of the trailing monomial to be eventually fixed (under the conditions of the theorem).
\end{rem}

This provides a satisfactory answer in the simply-laced case. However, the non-simply-laced case is tricky, since the coloured Jones polynomial is not a truncation of the corresponding $\sW$ algebra characters. Here, we have to resort to some direct calculations, as the following theorem shows.

\begin{thm}\label{thm:minexp}
Let $\overline{\la{g}}$ be any finite-dimensional simple Lie algebra, not necessarily simply-laced.
Let
\begin{align*}
	\lambda\in \overline P_+\cap \overline Q 
\end{align*}
and let $d$ be the length of short roots in $\la{g}$.
Suppose that $p,p'$ are non-negative co-prime integers such that
\begin{align}
	\frac{1}{p}+\frac{1}{p'} < \dfrac{d}{2\|\overline\rho\|}.
	\label{cond:pp'large}
\end{align}			
Then the trailing monomial for $\jones_{T(p,p')}(\lambda)$ is $m_{\lambda}(0)q^{0}$. The only contribution to this term in the equation \eqref{eqn:jones} is due to the summand corresponding to $\mu=0\in \wt(\lambda)$ and $\overline w=1$. We thus have that
	\begin{align*}
	&\widehat{\jones}_{T(p,p')}(\lambda)
	=\dfrac{q^{-\frac{pp'}{2}\left(\frac{1}{p}-\frac{1}{p'}\right)^2\|\overline\rho\|^2} }{\prod_{\alpha\succ 0 }(1-q^{(\alpha,\overline\rho)})}
	\sum_{\substack{\mu\in \wt(\lambda)\\ \overline w\in \overline W}}\frac{m_{\lambda}(\mu)}{m_{\lambda}(0)}\cdot (-1)^{\len(\overline w)} \cdot q^{\frac{pp'}{2}\left\| \mu +\frac{\overline \rho}{p}-\overline{w}\frac{\overline\rho}{p'}\right\|^2}.
	\end{align*}
\end{thm}
\begin{proof}
	Fix $\lambda\in \overline P_+\cap \overline Q, p,p'$ as in the statement of the theorem.
	Note that $\wt(\lambda)\subseteq \overline Q$.
	
	For any $\overline w\in \overline W$, let us denote
	\begin{align*}
		\mu_{\overline w} = - \frac{\overline \rho}{p} + \overline w \frac{\overline \rho}{p'}.
	\end{align*}
	
	Fix $\overline w \in \overline W$. We claim that under the given conditions on $p,p'$, the closest element of $\overline Q$ to $\mu_{\overline w}$ is $0$.
	We have that:
	\begin{align}
		\|\mu_{\overline w}\|^2= \left(\frac{1}{p^2} + \frac{1}{p'^2} \right)\|\overline\rho\|^2 - \frac{2}{pp'}(\overline \rho, \overline w\, \overline \rho)\leq \left(\frac{1}{p}+\frac{1}{p'}\right)^2\|\overline\rho\|^2
		< \frac{d^2}{4}.
		\label{eqn:lenmuw}
	\end{align}
	Note that for all  $\mu\in \overline {Q}\backslash \{0\}$, $\|\mu\| \geq d$ (this follows from \cite[Prop.\ 5.10]{Kac-book}). Thus, we have, for all $\mu\in \overline Q\backslash \{0\}$:	
	\begin{align*}
		\|\mu-\mu_{\overline w}\|\geq \|\mu\| -\|\mu_{\overline w}\|>\frac{d}{2}.
	\end{align*}
	This means that the minimum of $\|\mu-\mu_{\overline w}\|^2$ as $\mu$ varies over $\overline Q$ is achieved uniquely when $\mu=0$, and the minimum equals $\|\mu_{\overline w}\|^2$.
	
	Now, vary $\overline w\in \overline W$. From \eqref{eqn:lenmuw}, it is clear that $\|\mu_{\overline w}\|^2$ is the smallest iff $\overline w \,\overline \rho =\overline \rho$. This happens iff $\overline w = 1$.
	
	Note that since $\lambda\in \overline P_+ \cap \overline Q$, $0\in \wt(\lambda)$. The analysis above shows that the trailing monomial in $\jones_{T(p,p')}(\lambda)$ in \eqref{eqn:jones} is uniquely contributed by the term corresponding to $\mu=0$ and $\overline w=1$. This immediately proves what we need.
\end{proof}	
\begin{rem}\label{rem:minexp}
	The bound \eqref{cond:pp'large} is quite crude when $\overline{\la{g}}$ is simply-laced. Even when $\overline{\la{g}}=\la{sl}_3$, $p=3, p'=4$, it fails, but these values are allowed in Theorem \ref{thm:ADEjones_hat} since $h^\vee=3$ for $\la{sl}_3$. 
\end{rem}

\subsection{Limits}
It is now time to consider large colour limits of coloured invariants of torus knots.

\begin{defi}
\label{def:relasympmul}
Let $\overline{\la{g}}$ be any finite-dimensional simple Lie algebra (not necessarily simply-laced), $\lambda \in \overline P_+\backslash \{0\}$. We say that $\lambda$ satisfies the \emph{relative asymptotic multiplicity condition} if for all $\mu_1, \mu_2 \in \lambda+\overline{Q}$ we have:
\begin{align}
	\lim_{\substack{n\rightarrow\infty \\ n\lambda\in \lambda+\overline{Q}}} \dfrac{m_{n\lambda}(\mu_1)}{m_{n\lambda}(\mu_2)} = 1.
	\label{eqn:relasympmul}
\end{align}	
Note that for small values of $n$, the quotient of multiplicities may be undefined or zero, but it eventually becomes well-defined and non-zero thanks to Lemma \ref{lem:nonzeromult}. Also, we must restrict to $n\lambda\in \lambda+\overline Q$, otherwise both quantities in the fraction are necessarily zero.

Conjecture \ref{conj:main} below speculates that all weights
$\lambda \in \overline P_+\backslash \{0\}$ satisfy this condition.
\end{defi}

\begin{thm}
	\label{thm:ADElimits}
	Let $\overline{\la{g}}$ be a rank $\ell$ finite-dimensional simple Lie algebra, $p,p'$ be coprime positive integers, and let $\lambda\in\overline{P}_+\backslash\{0\}$.
	
	\begin{enumerate}
		\item If $\overline{\la{g}}$ is simply-laced, $p,p'\geq h^\vee$, and $0\neq \lambda\in \overline{\Lambda}_j+\overline{Q}$ ($j\in J$) satisfies the relative asymptotic multiplicity condition \eqref{eqn:relasympmul},
		 then we have:
		\begin{align}
			&\lim_{\substack{n\rightarrow\infty\\ n\lambda\in \overline{\Lambda}_j+\overline{Q}}}
			\widehat{\jones}_{T(p,p')}(n\lambda) \notag\\
			&\,\,= 
			\dfrac{
			q^{-\frac{pp'}{2}\left\|\overline\Lambda_j +\overline{\sigma}_j\frac{\overline{\rho}}{p}
			-\frac{\overline{\rho}}{p'} \right\|^2} }
			{\prod_{\alpha\succ 0 	}(1-q^{(\alpha,\overline\rho)})}
			\eta(q)^{\ell} \ch \mathbf{L}
			(\gamma_{-w_{0}\overline{\Lambda( (p-h^{\vee})\Lambda_j, (p'-h)\Lambda_0)}})\notag\\
			&\,\,\in 1+q\ZZ[[q]].
			\label{eqn:SClimit}
		\end{align}
		\item If $\overline{\la{g}}$ is not simply-laced, if $p,p'$ satisfy \eqref{cond:pp'large}, and if 
		$\lambda\in \overline{P}_+\cap \overline{Q}$ ($\lambda\neq 0$) satisfies the relative asymptotic multiplicity condition \eqref{eqn:relasympmul}, then we have:
		\begin{align}
			&\lim_{\substack{n\rightarrow\infty}}
			\widehat{\jones}_{T(p,p')}(n\lambda) \notag\\
			&\,\,		=\dfrac{q^{-\frac{pp'}{2}\left(\frac{1}{p}-\frac{1}{p'}\right)^2\|\overline\rho\|^2} }{\prod_{\alpha\succ 0 }(1-q^{(\alpha,\overline\rho)})}
			\sum_{\substack{\mu\in \overline Q\\ \overline w\in \overline W}} (-1)^{\len(\overline w)} \cdot q^{\frac{pp'}{2}\left\| \mu +\frac{\overline \rho}{p}-\overline{w}\frac{\overline\rho}{p'}\right\|^2}.
			\label{eqn:nonSClimit}
		\end{align}
	\end{enumerate} 
\end{thm}
\begin{proof} 
	The statements are clear after combining the relative asymptotic multiplicity condition, Corollary \ref{cor:limwtset}, and the characters of $\sW$ algebra modules given in \eqref{eqn:wchar3}.
\end{proof}

\begin{rem}
	The pure power of $q$ appearing in \eqref{eqn:SClimit} depends on $p,p'$. However, it is natural to absorb this factor (along with the $q^{\ell/24}$ contributed by $\eta(q)^\ell$) into $\ch \mathbf{L}$. This turns $\ch\mathbf{L}$ into the normalized character, say, $\overline{\ch}\, \mathbf{L}\in 1+q\ZZ_{\geq 0}[[q]]$, which is better suited for combinatorial purposes. With this, the right-hand side of \eqref{eqn:SClimit} may be written as:
	\begin{align*}
		\dfrac{(q)_\infty^\ell}
		{\prod_{\alpha\succ 0}(1-q^{(\alpha,\overline\rho)})}
		\cdot 
		\overline{\ch}\,\mathbf{L}
		(\gamma_{-w_{0}\overline{\Lambda( (p-h^{\vee})\Lambda_j, 	(p'-h)\Lambda_0)}}).
	\end{align*}
\end{rem}

\begin{rem} 
	This theorem subsumes most of the results from our previous paper \cite{Kan-torus} since the $\la{sl}_r$ highest weights $n\overline{\Lambda}_1$ satisfy \eqref{eqn:relasympmul}. These are highest weights of the irreducible modules $\mathrm{Sym}^n(\CC^r)$ where each weight multiplicity is exactly $1$.
\end{rem}

\begin{rem}\label{rem:nonSClimit}
	In \eqref{eqn:nonSClimit}, if $\overline{\rho}/p'$ were $\overline{\rho^\vee}/p'$ and if $\overline{Q}$ were changed to $\overline{Q}^\vee$, it would be the character of the corresponding $\sW$ algebra. However, the non-simply laced case prevents this from happening. The LHS of \eqref{eqn:nonSClimit} is always invariant under $p\longleftrightarrow p'$ (Remark \ref{rem:pp'invariance}), however in the non-simply-laced cases, the tricky nature of Feigin--Frenkel duality prevents the $\sW$ algebras at shifted levels $p/p'$ and $p'/p$ from being isomorphic.
\end{rem}

In the next section, we will give examples of Lie algebras and their weights which satisfy \eqref{eqn:relasympmul}.

\section{Relative asymptotic weight multiplicities}
\label{sec:relasympmul}

As we have shown, matching the large colour limits of coloured Jones polynomials hinges upon the relative asymptotic multiplicities being $1$. Our main conjecture is that this is always the case.

\begin{conj}
	\label{conj:main}
	For all finite-dimensional simple Lie algebras $\overline{\la{g}}$ and all non-zero dominant integral weights $\lambda\in \overline{P}_+\backslash \{0\}$, the relative asymptotic multiplicity condition	\eqref{eqn:relasympmul} is satisfied.
\end{conj}

We now check the validity of this conjecture in various examples.

\subsection{
	\texorpdfstring{Type $\lietype{A}_r$, $\lambda=j\overline{\Lambda}_1$}{Type Ar, first fundamental weight}}
The multiplicity of each weight in $L(j\overline \Lambda_1)$ module for $\la{sl}_r$ is $1$; this is just the $j$th symmetric power of the defining representation of $\la{sl}_r$.
Thus, clearly \eqref{eqn:relasympmul} holds for this family of modules, and we have already connected the relevant limits of coloured invariants to principal $\sW$-algebra characters in \cite{Kan-torus}.

\subsection{
	\texorpdfstring{Type $\lietype{C}_r$, $\lambda=j\overline\Lambda_1$}
	{Type Cr, first fundamental weight}
}

Here, the full character of the module is given by the appropriately specialized complete homogeneous symmetric polynomial (see \cite[Ch.\ 24]{FulHar}):
\begin{align*}
	H_j(x_1,\cdots,x_r,x_1^{-1},\cdots,x_r^{-1}).
\end{align*}
For a fixed weight $\mu\in \overline{P}$, the coefficient of $x_1^{\mu_1}\cdots x_r^{\mu_r}$ is the multiplicity $m_{j\overline{\Lambda}_1}(\mu)$ and is given by the number of integral solutions of the system:
\begin{align*}
	a_1+\cdots+a_r+b_1+\cdots+b_r = j,\\
	a_1-b_1=\mu_1,\cdots,a_r-b_r=\mu_r,\\
	a_1, \cdots, a_r, b_1,\cdots, b_r\geq 0.
\end{align*}
This is the same as the number of integral solutions to:
\begin{align*}
	2a_1+\cdots+2a_r = j+\mu_1+\cdots+\mu_r,\\
	a_1,\cdots, a_r\geq 0,\quad 
	a_1\geq \mu_1,\cdots, a_r\geq \mu_r.
\end{align*}
Let $\tilde\mu_i = \max(0,\mu_i)$ and $a_i=\tilde a_i + \tilde \mu_i$ for $1\leq i\leq r$. Then, the system is equivalent to:
\begin{align*}
	2\tilde a_1+\cdots+2\tilde a_r &= j+\mu_1+\cdots+\mu_r - 2\tilde \mu_1 - \cdots -2\tilde \mu_r = j-|\mu_1|-\cdots-|\mu_r|,\\
	\tilde a_1,\cdots, \tilde a_r&\geq 0.
\end{align*}
Let us assume that $j-|\mu_1|-\cdots-|\mu_r|$ is even, which is the same as assuming that $\mu\in j\overline{\Lambda}_1+ \overline Q$.
Further assume that $j\ge |\mu_1|+\cdots +|\mu_r|$.
Under these restrictions, 
\begin{align*}
	m_{j\overline \Lambda_1}(\mu) = 
	\binom{\frac{1}{2}(j-|\mu_1|-\cdots-|\mu_r|) + r-1}{r-1}
	= \frac{j^{r-1}}{2^{r-1}(r-1)!} + \cdots.
\end{align*}
Crucially, the highest order term is independent of $\mu$ (but the lower order terms in $j$ depend on $\mu$). 
This implies that \eqref{eqn:relasympmul} holds for this family of modules. In particular, one may take $\lambda=2\overline{\Lambda}_1 \in \overline{P}_+\cap \overline{Q}$ in Theorem \eqref{thm:ADElimits} part (2) for $\overline{\la{g}}$ of type $\lietype C_r$.

\subsection{
	\texorpdfstring{Type $\lietype{B}_r$, $\lambda=j\overline \Lambda_1$}
	{Type Br, first fundamental weight}
}

Here, the full character of the module is (see \cite[Ch.\ 24]{FulHar}):
\begin{align*}
	H_j(x_1,\cdots,x_r,x_1^{-1},\cdots,x_r^{-1},1)-H_{j-2}(x_1,\cdots,x_r,x_1^{-1},\cdots,x_r^{-1},1).
\end{align*}
Note that here $j\overline{\Lambda}_1\in \overline Q\cap \overline P_+$.
Thus, we restrict to $\mu\in \overline Q$.
Given such a $\mu$, the coefficient of $x_1^{\mu_1}\cdots x_r^{\mu_r}$ ($\mu_1,\cdots,\mu_r\in\ZZ$)  in the first term ($H_j$) is the number of non-negative integral solutions (for $a_1,\cdots,a_r,b_1,\cdots,b_r,c$) of:
\begin{align*}
	a_1+\cdots+a_r+b_1+\cdots+b_r +c = j\\
	a_1-b_1=\mu_1,\cdots,a_r-b_r=\mu_r,
\end{align*}
equivalently, integral solutions of:
\begin{align*}
	2a_1+\cdots+2a_r+c&=j+\mu_1+\cdots+\mu_r,\\
	a_1,\cdots, a_r\geq 0,\,\, &a_1\geq \mu_1,\cdots a_r\geq \mu_r,\\
	c&\geq 0.
\end{align*}
Now, the parity of $c$ is determined by the parity of $j+\mu_1+\cdots+\mu_r$. Let $\epsilon=1$ if this parity is odd, and $\epsilon=0$ otherwise, and let $c=2\tilde c+\epsilon$. 
Proceeding as in the $\lietype{C}_r$ case, the system has the same number of solutions as the system:
\begin{align*}
	2\tilde a_1+\cdots+2\tilde a_r+2\tilde c&=j-|\mu_1|-\cdots-|\mu_r|-\epsilon,\\
	\tilde a_1,\cdots, \tilde a_r&\geq 0,\,\,\tilde c\geq 0.
\end{align*}
Now for all $j\geq |\mu_1|+\cdots + |\mu_r|+\epsilon$, the number of solutions is
\begin{align*}
	\binom{\frac{1}{2}(j-|\mu_1|-\cdots-|\mu_r| -\epsilon)+ r}{r}.
\end{align*}
Similarly we analyze the $H_{j-2}$ term, which requires $j-2\geq |\mu_1|+\cdots + |\mu_r|+\epsilon$. All in all,
for large enough $j$ (depending on $\mu$),
the required multiplicity  is then:
\begin{align*}
	&m_{j\overline\Lambda_1}(\mu)\\
	&=\binom{\frac{1}{2}(j-|\mu_1|-\cdots-|\mu_r| -\epsilon) + r}{r}- \binom{\frac{1}{2}(j-2-|\mu_1|-\cdots-|\mu_r| -\epsilon) + r}{r}\\
	&=\binom{\frac{1}{2}(j-|\mu_1|-\cdots-|\mu_r| -\epsilon) + r-1}{r-1}
	=\frac{j^{r-1}}{2^{r-1}(r-1)!}+\cdots.
\end{align*}
Again, we see that the dominant term is independent of $\mu$ (lower order terms depend on $\mu$). This implies that \eqref{eqn:relasympmul} holds for this family of modules.
In particular, one may take $\lambda=\overline \Lambda_1\in \overline{P}_+\cap \overline{Q}$ in Theorem \ref{thm:ADElimits} part (2) for $\overline{\la{g}}$ of type $\lietype B_r$.

\subsection{
	\texorpdfstring{Type $\lietype{D}_r$, $\lambda=j\overline\Lambda_1$}
	{Type Dr, first fundamental weight}
}	

Here, the full character of the module is given as follows:
\begin{align*}
	H_j(x_1,\cdots,x_r,x_1^{-1},\cdots,x_r^{-1})-H_{j-2}(x_1,\cdots,x_r,x_1^{-1},\cdots,x_r^{-1}).
\end{align*}
Given a weight $\mu$, the coefficient of $x_1^{\mu_1}\cdots x_r^{\mu_r}$ ($\mu_1,\cdots,\mu_r\in\ZZ$) in the first term is the number of non-negative integral solutions (for $a_1,\cdots,a_r,b_1,\cdots,b_r$) of:
\begin{align*}
	a_1+\cdots+a_r+b_1+\cdots+b_r = j\\
	a_1-b_1=\mu_1,\cdots,a_r-b_r=\mu_r.
\end{align*}
Proceeding as in the $\lietype{C}_r$ case, assuming that $j-|\mu_1|-\cdots-|\mu_r|$ is even (equivalently, $\mu\in j\overline{\Lambda}_1+\overline Q$), and that $j\geq |\mu_1|+\cdots+|\mu_r|$, the number of solutions is:
\begin{align*}
	\binom{\frac{1}{2}(j-|\mu_1|-\cdots-|\mu_r|) + r-1}{r-1}.
\end{align*}
Analyzing similarly the $H_{j-2}$ term, (this time assuming that 
$j\geq |\mu_1|+\cdots+|\mu_r|+2$), the required coefficient is
\begin{align*}
	\binom{\frac{1}{2}(j-2-|\mu_1|-\cdots-|\mu_r|) + r-1}{r-1}.
\end{align*}
Combining, for $j\gg 0$ (depending on $\mu$) and of requisite parity,
\begin{align*}
	&m_{j\overline\Lambda_1}(\mu)\\
	&=\binom{\frac{1}{2}(j-|\mu_1|-\cdots-|\mu_r|) + r-1}{r-1}- \binom{\frac{1}{2}(j-2-|\mu_1|-\cdots-|\mu_r|) + r-1}{r-1}\\
	&=\binom{\frac{1}{2}(j-|\mu_1|-\cdots-|\mu_r|) + r-2}{r-2}
	=\frac{j^{r-2}}{2^{r-2}(r-2)!}+\cdots.
\end{align*}
We see again that the dominant term is independent of $\mu$. This implies that \eqref{eqn:relasympmul} holds for this family of modules.

In Theorem \ref{thm:ADElimits} part (1) for $\overline{\la{g}}$ of type $\lietype D_r$, we may take $\lambda=\overline{\Lambda}_1$, in which case, one has to progress through odd values of $n$ while calculating the limit.
If we take $\lambda=2\overline{\Lambda}_1\in \overline{P}_+\cap\overline{Q}$, all values of $n$ are allowed.

\subsection{Generalities about weight multiplicities}
In the next few subsections, we will show that the relative asymptotic multiplicities in rank $2$ Lie algebras satisfy \eqref{eqn:relasympmul}. 
In order to achieve this, we will use
Kostant's multiplicity formula \cite[Ch.\ 24]{Hum-la}
\begin{align*}
	m_{\lambda}(\mu)=\sum_{\overline w\in\overline  W}(-1)^{\len(\overline w)} \wp(\overline w\lambda + \overline w\,\overline \rho - \mu-\overline \rho),
\end{align*}
where $\wp$ is the Kostant partition function, which is a 
kind of a vector partition function. There is a beautiful theory about the behavior of such functions, see for example \cite{Stu} and \cite{BriVer-vectorpartitions}, etc. Vector partition functions are known to be piece-wise quasi-polynomial functions, a fact that was used by Garoufalidis and Vuong \cite{GarVuo}. The domains of quasi-polynomiality depend on how the space $\overline{Q}_+ = \ZZ_{\geq 0}\alpha_1+\cdots+\ZZ_{\geq 0}\alpha_\ell$ is partitioned into ``chambers'' (not to be confused with the Weyl chambers). In the rank $2$ cases we will analyze below, the chambers are easily described, but for higher ranks, their description is fairly involved; see for example \cite{BilGuiRas}. In any case, we will see in the examples below that there is a marked difference between the growth of the functions $n\mapsto m_{n\lambda}(\mu)$ when the weight $\lambda$ lies on a hyperplane spanned by a subset of positive roots and when it does not.

We mention in passing that exact formulas for the Kostant partition function are also required in relation to modular invariance properties of 
characters of certain non-rational VOAs, \cite{BriKasMilNaz-higherfalse}.

\subsection{Multiplicities in \texorpdfstring{$\lietype{A}_2$}{A2}}
The behaviour of multiplicities in $\lietype A_2$ is very well-known, see for example \cite{BilGuiRas}.

First let $\lambda$ not be proportional to $\overline \rho$. In this case, there is a triangular lacunary region (see all except the third picture in \eqref{fig:A2weights}) inside $\wt(\lambda)$ where the multiplicities are all equal to each other. In fact, if $\lambda$ is proportional to $\overline \Lambda_1$ or $\overline \Lambda_2$, then all of $\wt(\lambda)$ falls in this lacunary region; see the first and the last pictures in \eqref{fig:A2weights}.
This lacunary region is determined by the intersections of lines joining certain Weyl translates of $\lambda$, and thus, this region grows with $n\lambda$ as $n$ grows. In particular, any two fixed $\mu_1$ and $\mu_2$ in $\lambda+\overline Q$ eventually fall inside this lacunary region for all large $n\gg 0$ (still stipulating that $n\lambda\in \lambda+\overline Q$). Thus, in this case, \eqref{eqn:relasympmul} clearly holds.

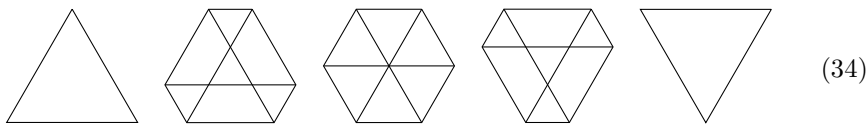
\begin{align}
	\label{fig:A2weights}
	\begin{matrix}
	\begin{tikzpicture}[scale=0.5]
		\draw (1,0) coordinate (a1);
		\draw (120:1) coordinate (a2);
		\draw (1,0) coordinate (a1);
		\draw (90:0.67) coordinate (l2);
		\draw (30:0.67) coordinate (l1);
		\draw ($3*(l2)$) coordinate (lambda1);
		\draw ($3*(l1)-3*(l2)$) coordinate (lambda2);
		\draw ($-3*(l1)$) coordinate (lambda3);
		\draw (lambda1) -- (lambda2) -- (lambda3) -- (lambda1);
	\end{tikzpicture}
	\quad
	\begin{tikzpicture}[scale=0.5]
		\draw (1,0) coordinate (a1);
		\draw (120:1) coordinate (a2);
		\draw (1,0) coordinate (a1);
		\draw (90:0.67) coordinate (l2);
		\draw (30:0.67) coordinate (l1);
		\draw ($(l1)+2*(l2)$) coordinate (lambda1);
		\draw ($3*(l1)-2*(l2)$) coordinate (lambda2);
		\draw ($2*(l1)-3*(l2)$) coordinate (lambda3);
		\draw ($-2*(l1)-1*(l2)$) coordinate (lambda4);
		\draw ($-3*(l1)+1*(l2)$) coordinate (lambda5);
		\draw ($-1*(l1)+3*(l2)$) coordinate (lambda6);
		\draw (lambda1) -- (lambda4);
		\draw (lambda2) -- (lambda5);
		\draw (lambda3) -- (lambda6);
		\draw (lambda1) -- (lambda2) -- (lambda3) -- (lambda4) -- (lambda5) -- (lambda6)--(lambda1);
	\end{tikzpicture}
	\quad
	\begin{tikzpicture}[scale=0.5]
		\draw (1,0) coordinate (a1);
		\draw (120:1) coordinate (a2);
		\draw (1,0) coordinate (a1);
		\draw (90:0.67) coordinate (l2);
		\draw (30:0.67) coordinate (l1);
		\draw ($1.5*(l1)+1.5*(l2)$) coordinate (lambda1);
		\draw ($3*(l1)-1.5*(l2)$) coordinate (lambda2);
		\draw ($1.5*(l1)-3*(l2)$) coordinate (lambda3);
		\draw ($-1.5*(l1)-1.5*(l2)$) coordinate (lambda4);
		\draw ($-3*(l1)+1.5*(l2)$) coordinate (lambda5);
		\draw ($-1.5*(l1)+3*(l2)$) coordinate (lambda6);
		\draw (lambda1) -- (lambda4);
		\draw (lambda2) -- (lambda5);
		\draw (lambda3) -- (lambda6);
		\draw (lambda1) -- (lambda2) -- (lambda3) -- (lambda4) -- (lambda5) -- (lambda6)--(lambda1);
	\end{tikzpicture}
	\quad
	\begin{tikzpicture}[scale=0.5]
		\draw (1,0) coordinate (a1);
		\draw (120:1) coordinate (a2);
		\draw (1,0) coordinate (a1);
		\draw (90:0.67) coordinate (l2);
		\draw (30:0.67) coordinate (l1);
		\draw ($2*(l1)+(l2)$) coordinate (lambda1);
		\draw ($3*(l1)-(l2)$) coordinate (lambda2);
		\draw ($1*(l1)-3*(l2)$) coordinate (lambda3);
		\draw ($-1*(l1)-2*(l2)$) coordinate (lambda4);
		\draw ($-3*(l1)+2*(l2)$) coordinate (lambda5);
		\draw ($-2*(l1)+3*(l2)$) coordinate (lambda6);
		\draw (lambda1) -- (lambda4);
		\draw (lambda2) -- (lambda5);
		\draw (lambda3) -- (lambda6);
		\draw (lambda1) -- (lambda2) -- (lambda3) -- (lambda4) -- (lambda5) -- (lambda6)--(lambda1);
	\end{tikzpicture}
	\quad
	\begin{tikzpicture}[scale=0.5]
		\draw (1,0) coordinate (a1);
		\draw (120:1) coordinate (a2);
		\draw (1,0) coordinate (a1);
		\draw (90:0.67) coordinate (l2);
		\draw (30:0.67) coordinate (l1);
		\draw ($3*(l1)$) coordinate (lambda1);
		\draw ($-3*(l1)+3*(l2)$) coordinate (lambda2);
		\draw ($-3*(l2)$) coordinate (lambda3);
		\draw (lambda1) -- (lambda2) -- (lambda3) -- (lambda1);
	\end{tikzpicture}
	\end{matrix}
\end{align}

If $\lambda$ is proportional to $\overline \rho$, then the multiplicity is maximum at the origin and it decreases by $1$ with every hexagonal shell around the origin (the $n$th shell is the boundary of the convex hull of $n\overline\Delta$).
In particular, for any fixed $\mu\in\overline  Q$,  $m_{n\overline \rho}(0)-m_{n\overline \rho}(\mu)$ is eventually a constant $C$ (eventually, because $n$ needs to be large enough for $\mu$ to be a weight of $L(n\overline \rho)$). This constant $C$ is the shell on which $\mu$ falls. Further, $m_{n\overline \rho}(0)=n+1$. Thus, in this case also,
\begin{align*}
	\lim_{n\rightarrow \infty}\dfrac{m_{n\overline \rho}(\mu)}{m_{n\overline \rho}(0)}= \lim_{n\rightarrow \infty}\dfrac{n+1-C}{n+1} = 1.
\end{align*}

Thus, for $\overline{\la{g}} = \la{sl}_3 = \lietype{A}_2$, we may take any $\lambda\in \overline{P}_+\backslash\{0\}$ in Theorem \ref{thm:ADElimits} part (1).

\subsection{Multiplicities in \texorpdfstring{$\lietype{C}_2$}{C2}}
In the case of $\lietype A_2$, we broke the cases into two: the weights proportional to $\overline \rho$ and the weights that are not. As mentioned above, in general, the break-up is between \emph{regular} and \emph{non-regular} weights; where a weight is regular if it does not fall on any hyperplane spanned by a subset of positive roots.

In case of $\lietype C_2$, the fundamental weights are $\overline \Lambda_1=\epsilon_1$ and $\overline \Lambda_2=\epsilon_1+\epsilon_2$  (with $(\epsilon_i,\epsilon_j)=\delta_{i,j}$).
Note that this normalization means that the long roots have squared length $4$, which is different from the form $(\cdot,\cdot)$ defined above, however, this will not make a difference below.
The simple roots are $\alpha_1=\epsilon_1-\epsilon_2$, $\alpha_2=2\epsilon_2$, and rest of the positive roots are $\alpha_1+\alpha_2 = \epsilon_1+\epsilon_2$ and 
$2\alpha_1+\alpha_2 = 2\epsilon_1$.
Thus, $\overline \rho=2\epsilon_1+\epsilon_2$. 
The regular weights $\lambda\in\overline  P_+$ are the ones not proportional to $\overline \Lambda_1=\epsilon_1$ or $\overline \Lambda_2=\epsilon_1+\epsilon_2$.
The Weyl group has order $8$.

Let us write down the shape of the function $\wp$. 
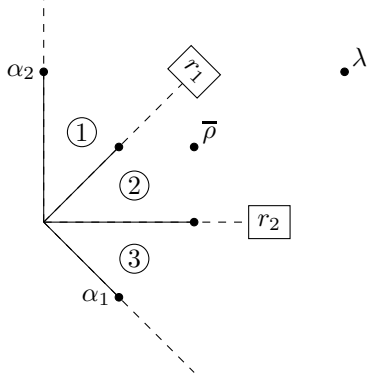
\begin{align*}
	\begin{tikzpicture}
		\draw (0,0) -- (1,-1);
		\draw (0,0) -- (0,2);
		\draw (0,0) -- (1,1);
		\draw (0,0) -- (2,0);
		\draw[dashed] (0,0) -- (2,-2);
		\draw[dashed] (0,0) -- (0,3);
		\draw[dashed] (0,0) -- (2,2);
		\draw[dashed] (0,0) -- (3,0);
		\node[draw, shape = circle, fill=black, minimum size = 0.1cm, inner sep=0pt] at (1,-1) {};
		\node[draw, shape = circle, fill=black,  minimum size = 0.1cm, inner sep=0pt] at (0,2) {};
		\node[draw, shape = circle, fill=black, minimum size = 0.1cm, inner sep=0pt] at (1,1) {};
		\node[draw, shape = circle, fill=black, minimum size = 0.1cm, inner sep=0pt] at (2,0) {};
		\node[draw, shape = circle, fill=black, minimum size = 0.1cm, inner sep=0pt] at (2,1) {};
		\node[draw, shape = circle, fill=black, minimum size = 0.1cm, inner sep=0pt] at (4,2) {};
		\node at (0.7,-1) {$\alpha_1$};
		\node at (-0.3,2) {$\alpha_2$};
		\node at (2.2,1.2) {$\overline \rho$};
		\node at (4.2,2.2) {$\lambda$};
		\node[draw, shape = circle,  inner sep=1pt] at (67:1.3) {1};
		\node[draw, shape = circle, inner sep=1pt] at (22:1.3) {2};
		\node[draw, shape = circle,  inner sep=1pt] at (-22:1.3) {3};
		\node[draw, fill=white, rotate=-45] at (2,2) {$r_1$};
		\node[draw, fill=white, rotate=-0] at (3,0) {$r_2$};
	\end{tikzpicture}
\end{align*}

The dominant terms in Kostant partition function for $\alpha=a_1\alpha_1 + a_2\alpha_2$ in each of the regions is given by \cite{tar-partition} (where $\cdots$ denote lower order terms):
\begin{align*}
	\wp(\alpha) = 
	\begin{cases}
		\frac{1}{4}a_1^2 +\cdots & \mathrm{if}\,\, \alpha \in {\circled{1}},\\
		-\frac{1}{4}a_1^2 + a_1a_2 - \frac{1}{2}a_2^2+\cdots & \mathrm{if}\,\, \alpha \in {\circled{2}},\\
		\frac{1}{2}a_2^2 +\cdots & \mathrm{if}\,\, \alpha \in {\circled{3}}
	\end{cases}
\end{align*}

Now, let $\lambda=a_1\alpha_1+a_2\alpha_2\in \overline{P}_+$   (i.e., it belongs to region {\circled{2}}), and not proportional to $\alpha_1+\alpha_2=\epsilon_1+\epsilon_2$ or $\alpha_1+2\alpha_2=2\epsilon_1$ (i.e., it is in the interior of region {\circled{2}}). Let $\mu=m_1\alpha_1+m_2\alpha_2$ be in $\lambda+\overline Q$ and we will always let $n$ such that $n\lambda\in\lambda + \overline  Q$.

For $n\gg 0$, we have the following containment:
\begin{align*}
	n\overline w\lambda +\overline w\,\overline \rho-\mu-\overline \rho= 
	\begin{cases}
		(2na_2-na_1-1-m_1)\alpha_1+(na_2-m_2) \alpha_2\in {\circled{1}} &\mathrm{if}\,\, \overline w=r_1,\\
		(na_1-m_1)\alpha_1+(na_2-m_2)\alpha_2\in {\circled{2}} &\mathrm{if}\,\, \overline w=1,\\
		(na_1  -m_1)\alpha_1 + (n(a_1-a_2)-1-m_2)\alpha_2\in {\circled{3}} &\mathrm{if}\,\, \overline w=r_2.
	\end{cases}
\end{align*}
while for other $\overline w$, $n\overline w\lambda +\overline w\,\overline \rho-\mu-\overline \rho $ will not be a non-negative linear combination of positive roots.
We thus have (for $n\gg 0$):
\begin{align*}
	m_{n\lambda}(\mu) &= \sum_{\overline w\in \{r_1,1,r_2\}}
	(-1)^{\len(\overline w)} \wp(n\overline w\lambda +\overline w\,\overline \rho-\mu-\overline \rho),\\
	&=-\wp( nr_1\lambda+r_1\overline \rho-\mu-\overline \rho)+\wp( n\lambda-\mu)-\wp( nr_2\lambda+r_2\overline \rho-\mu-\overline \rho)\\
	&=- \frac{1}{4} (2na_2-na_1-1-m_1)^2+\cdots\\
	&\quad 
	-\frac{1}{4}(na_1-m_1)^2+(na_1-m_1)(na_2-m_2)-\frac{1}{2}(na_2-m_2)^2+\cdots\\
	&\quad - \frac{1}{2} (na_1-na_2-1-m_2) ^2+\cdots\\
	&= n^2\left( -\frac{1}{4}a_1^2+a_1a_2-\frac{1}{2}a_2^2 -\frac{1}{4}(2a_2-a_1)^2 - \frac{1}{2}(a_1-a_2)^2 \right) + \cdots\\
	&=(a_1-a_2)(2a_2-a_1)n^2+\cdots,
\end{align*}
where $\cdots$ in the final line correspond to lower order terms in $n$.
Clearly, this dominant term (in $n$) only depends on $\lambda$, and is non-zero as long as $\lambda$ is in the interior of region {\circled{2}}.

We thus conclude that if $\lambda=a_1\alpha_1+a_2\alpha_2\in \overline{P}_+$ is in the interior of region \circled{2}, $\mu_1,\mu_2\in \lambda+\overline Q$ and $n\gg 0$ (depending on $\mu_1$ and $\mu_2$ both), then,
\begin{align*}
	\lim\limits_{\substack{n\rightarrow\infty\\ n\lambda\in\lambda+\overline  Q}} \dfrac{m_{n\lambda}(\mu_1)}{m_{n\lambda}(\mu_2)}
	=\lim\limits_{\substack{n\rightarrow\infty\\ n\lambda\in \lambda+\overline Q}} \dfrac{n^2(-(a_1-a_2)(a_1-2a_2))+\cdots}{n^2(-(a_1-a_2)(a_1-2a_2))+\cdots} = 1,
\end{align*}
note again that the lower order terms $\cdots$ do depend on $\mu_1$ and $\mu_2$, but obviously they do not matter in the limit.

If $\lambda$ is proportional to either $\overline \Lambda_1=2\alpha_1+\alpha_2$ or $\overline \Lambda_2=\alpha_1+\alpha_2$, the result still holds by using the behavior of the multiplicities of relevant modules explained in \cite[Sec.\ 16.2]{FulHar}.

Thus, for $\overline{\la{g}} = \la{sp}_4 = \lietype{C}_2$, we may take any $\lambda\in (\overline{P}_+\cap \overline{Q} )\backslash\{0\}$ in Theorem \ref{thm:ADElimits} part (2).
\subsection{Multiplicities in \texorpdfstring{$\lietype{G}_2$}{G2}}
We will carry out analogous analysis for $\lietype{G}_2$, however it is somewhat more tedious. Note that in this case, $\overline{P}=\overline{Q}$ (thus, $\overline P_+\cap \overline Q = \overline P_+$).

The Kostant chamber, positive roots, fundamental roots, are as below.
Note that $\overline \rho=5\alpha_1+3\alpha_2$, $r_1\alpha_2=3\alpha_1+\alpha_2$ and $r_2\alpha_1=\alpha_1+\alpha_2$.
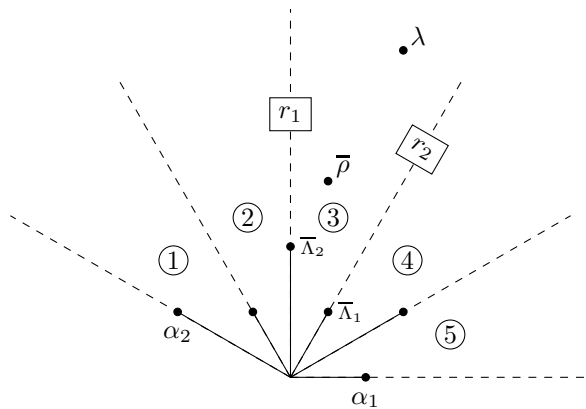
\begin{align*}
	\begin{tikzpicture}
		\draw (1,0) coordinate (a1);
		\draw (-1.5,0.87) coordinate (a2);
		\draw (0,0) -- (a1);
		\draw (0,0) -- (a2);
		\draw (0,0) -- ($(a1)+(a2)$);
		\draw (0,0) -- ($2*(a1)+(a2)$);
		\draw (0,0) -- ($3*(a1)+(a2)$);
		\draw (0,0) -- ($3*(a1)+2*(a2)$);
		\node[draw, shape = circle, fill=black, minimum size = 0.1cm, inner sep=0pt] at (a1) {};
		\node[draw, shape = circle, fill=black, minimum size = 0.1cm, inner sep=0pt] at (a2) {};
		\node[draw, shape = circle, fill=black, minimum size = 0.1cm, inner sep=0pt] at ($(a1)+(a2)$) {};
		\node[draw, shape = circle, fill=black, minimum size = 0.1cm, inner sep=0pt] at ($2*(a1)+(a2)$) {};
		\node[draw, shape = circle, fill=black, minimum size = 0.1cm, inner sep=0pt] at ($3*(a1)+(a2)$) {};
		\node[draw, shape = circle, fill=black, minimum size = 0.1cm, inner sep=0pt] at ($3*(a1)+2*(a2)$) {};
		\node[draw, shape = circle, fill=black, minimum size = 0.1cm, inner sep=0pt] at ($5*(a1)+3*(a2)$) {};
		\node at ($3*(a1)+2*(a2)+(0.3,-0.0)$) {$\scriptstyle{\overline \Lambda_2}$};
		\node at ($2*(a1)+(a2)+(0.3,-0.0)$) {$\scriptstyle{\overline \Lambda_1}$};
		\draw[dashed] (0,0) -- (0:4);
		\draw[dashed] (0,0) -- (30:4.3);
		\draw[dashed] (0,0) -- (60:4.6);
		\draw[dashed] (0,0) -- (90:4.9);
		\draw[dashed] (0,0) -- (120:4.6);
		\draw[dashed] (0,0) -- (150:4.3);
		\node at ($5*(a1)+3*(a2)+(0.2,0.2)$) {$\overline \rho$};
		\node at ($(a1)-(0,0.3)$) {$\alpha_1$};
		\node at ($(a2)-(0,0.3)$) {$\alpha_2$};
		\node[draw, shape = circle,  inner sep=1pt] at (135:2.2) {1};
		\node[draw, shape = circle, inner sep=1pt] at (105:2.2) {2};
		\node[draw, shape = circle, inner sep=1pt] at (75:2.2) {3};
		\node[draw, shape = circle, inner sep=1pt] at (45:2.2) {4};
		\node[draw, shape = circle, inner sep=1pt] at (15:2.2) {5};
		\node[draw, shape = circle, fill=black, minimum size = 0.1cm, inner sep=0pt] at ($9*(a1) + 5*(a2)$) {};
		\node at ($9*(a1) + 5*(a2)+(0.2,0.2)$) {$\lambda$};
		\node[draw, fill=white] at (90:3.5) {$r_1$};
		\node[draw, fill=white, rotate=-30] at (60:3.5) {$r_2$};
	\end{tikzpicture}
\end{align*}

The dominant terms in Kostant partition function for $\alpha=a_1\alpha_1 + a_2\alpha_2$ in each of the regions are (see \cite{tar-partition}) (where $\cdots$ denote lower order terms):
\begin{align*}
	\wp(\alpha) = 
	\begin{cases}
		\frac{1}{432}a_1^4 +\cdots & \mathrm{if}\,\, \alpha \in \circled{1},\\
		\frac{1}{432}a_1^4 -\frac{1}{48}(a_1-a_2-1)^4+\cdots & \mathrm{if}\,\, \alpha \in {\circled{2}},\\
		\frac{1}{48}a_2^4-\frac{1}{432}(3a_2-a_1-1)^4 + \frac{1}{48}(2a_2-a_1-2)^4 +\cdots & \mathrm{if}\,\, \alpha \in {\circled{3}},\\
		\frac{1}{48}a_2^4 - \frac{1}{432}(3a_2-a_1-1)^4 +\cdots & \mathrm{if}\,\, \alpha \in {\circled{4}},\\
		\frac{1}{48}a_2^4 +\cdots & \mathrm{if}\,\, \alpha \in {\circled{5}}.
	\end{cases}
\end{align*}
Let $\lambda=a_1\alpha_1+a_2\alpha_2 \in \overline{P}_+$ and not proportional to $\overline \Lambda_1=2\alpha_1+\alpha_2$ or $\overline \Lambda_2=3\alpha_1+2\alpha_2$ (i.e., it belongs to the interior of the region {\circled{3}}). Let $\mu=m_1\alpha_1+m_2\alpha_2$ be in $\overline Q$.
As long as $n\gg 0$, we have the following containment:
\begin{align*}
	&n\overline w\lambda +\overline w\,\overline \rho-\mu-\overline \rho\\
	&= 
	\begin{cases}
		(2a_1n - 3a_2n - m_1 - 4)\alpha_1 
		+ (a_1n - a_2n - m_2 - 1)\alpha_2
		\in{\circled{1}} &\mathrm{if}\,\, \overline w=r_1r_2,\\
		(-a_1n + 3a_2n - m_1 - 1)\alpha_1
		+(a_2n - m_2)\alpha_2
		\in
		{\circled{2}} &\mathrm{if}\,\,\overline  w=r_1,\\
		(a_1n - m_1)\alpha_1
		+(a_2n - m_2)\alpha_2
		\in{\circled{3}} &\mathrm{if}\,\,\overline  w=1,\\
		(a_1n - m_1)\alpha_1
		+(a_1n - a_2n - m_2 - 1)\alpha_2
		\in{\circled{4}} &\mathrm{if}\,\,\overline  w=r_2,\\
		(-a_1n + 3a_2n - m_1 - 1)\alpha_1
		+(-a_1n + 2a_2n - m_2 - 2)\alpha_2
		\in{\circled{5}} &\mathrm{if}\,\,\overline  w=r_2r_1.
	\end{cases}
\end{align*}
After a painstaking computation, we therefore have:
\begin{align*}
	m_{n\lambda}(\mu)& = \sum_{\overline w\in \{r_1r_2,r_1,1,r_2,r_2r_1\}}(-1)^{\len(\overline w)} \wp(n\overline w\lambda +\overline w\,\overline \rho-\mu-\overline \rho),\\
	&=\frac{1}{72}(a_1 - 2a_2)(2a_1 - 3a_2)(4a_1^2 - 15a_1a_2 + 12a_2^2)n^4+\cdots.
\end{align*}
The dominant term corresponds to $n^4$ and it is independent of $\mu$ (in fact, so is the case with $n^3$ term, but we have omitted it here) and non-zero whenever $\lambda$ is in the interior of ${\circled{3}}$, thus, again, we may conclude:
\begin{align*}
	\lim\limits_{\substack{n\rightarrow\infty}} \dfrac{m_{n\lambda}(\mu_1)}{m_{n\lambda}(\mu_2)}=1.
\end{align*}

We now handle the remaining cases $\lambda=a\overline \Lambda_1$ or $\lambda=a\overline \Lambda_2$. Here we will need the description of the sub-dominant terms that arise in the formula for $\wp$ given in \cite{tar-partition}. This is tedious, as expected. We omit most of the calculations.

Let us consider $\lambda=\overline \Lambda_2=3\alpha_1+2\alpha_2$. 
It is enough to consider $\mu\in\overline  P_+$, since every weight can be brought to this region using Weyl group action without changing its multiplicity.

Since $\mu\in\overline  P_+\subseteq\overline  Q_+=\ZZ_{\geq 0}\alpha_1+\ZZ_{\geq 0}\alpha_2$ and since $\overline w\,\overline \rho-\overline \rho\in -\overline  Q_+$ for all $\overline w$, we have that $\overline w\,\overline \rho-\overline \rho-\mu\in -\overline  Q_+$. We may thus safely omit those $\overline w$ such that $\overline w\overline \Lambda_2\in -\overline  Q_+\backslash \{0\}$ since $\wp(n\overline w\overline \Lambda_2+\overline w\,\overline \rho-\mu-\overline \rho)=0$.

This leaves us with $\overline w\in \{1,r_1,r_2,r_2r_1,r_1r_2,r_1r_2r_1\}$.
For the last two, $\overline w\overline \Lambda_2$ is on the outer boundary of region {\circled{1}}, and $r_1r_2\overline \rho-\overline \rho = -4\alpha_1-\alpha_2$, $r_1r_2r_1\overline \rho-\overline \rho=-6\alpha_1-2\alpha_2$, 
thus $\overline w\overline \Lambda_2+\overline w\overline \rho-\mu-\overline \rho\in -\overline Q_+\backslash\{0\}$.
We are thus left with $\overline w\in \{1,r_1,r_2,r_2r_1\}.$
With fixed $\mu=m_1\alpha_1+m_2\alpha_2$ and $n\gg 0$, we have the containment:
\begin{align*}
	&n\overline w\overline \Lambda_2+\overline w\,\overline \rho-\mu-\overline \rho\\
	&=
	\begin{cases}
		(3n - m_1)\alpha_1
		+(2n - m_2)\alpha_2
		\in{\circled{2}} &\mathrm{if}\,\, \overline w=1,\\
		(3n - m_1 - 1)\alpha_1
		+(2n - m_2)\alpha_2
		\in
		{\circled{2}} &\mathrm{if}\,\, \overline w=r_1,\\
		(3n - m_1)\alpha_1
		+(n - m_2 - 1)\alpha_2
		\in{\circled{5}} &\mathrm{if}\,\, \overline w=r_2,\\
		(3n - m_1 - 1)\alpha_1
		+(n - m_2 - 2)\alpha_2
		\in{\circled{5}} &\mathrm{if}\,\,\overline  w=r_2r_1.
	\end{cases}
\end{align*}

Now, one has to use the description of the sub-dominant terms in the partition function given in \cite{tar-partition}. Omitting the details, we now obtain:
\begin{align*}
	m_{n\overline \Lambda_2}(\mu)=\frac{n^3}{12} +\frac{3n^2}{8}+\cdots,
\end{align*}		
where the lower order terms depend on $\mu$, and also exhibit quasi-polynomiality. 
This establishes \eqref{eqn:relasympmul} for $\lambda=\overline \Lambda_2$.
Note also that when $\lambda$ was regular, the growth of the multiplicities in $L(n\lambda)$ was of the fourth power in $n$, but now, there is a cancellation in the highest order term, and we are left with a slower growth.

We work similarly for $\lambda=\overline \Lambda_1=2\alpha_1+\alpha_2$. In this case, the relevant $\overline w$s give us for $n\gg 0$:
\begin{align*}
	&n\overline w\Lambda_1+\overline w\,\overline \rho-\mu-\overline \rho\\
	&=
	\begin{cases}
		(2n - m_1)\alpha_1
		+(n - m_2)\alpha_2
		\in{\circled{4}} &\mathrm{if}\,\, \overline w=1,\\
		(2n - m_1)\alpha_1
		+(n - 1 - m_2)\alpha_2
		\in
		{\circled{4}} &\mathrm{if}\,\, \overline w=r_2,\\
		(n - 1 - m_1)\alpha_1
		+(n - m_2)\alpha_2
		\in{\circled{1}} &\mathrm{if}\,\, \overline w=r_1,\\
		(n - 4 - m_1)\alpha_1
		+(n - 1 - m_2)\alpha_2
		\in{\circled{1}} &\mathrm{if}\,\,\overline  w=r_1r_2.
	\end{cases}
\end{align*}
Adding all the contributions, we get:
\begin{align*}
	m_{n\overline \Lambda_1}(\mu)=\frac{n^3}{36} + \frac{5n^2}{24}+\cdots,
\end{align*}
Again, even though there is a cancellation of the $n^4$ terms, the next dominant term which corresponds to $n^3$ is independent of $\mu$ (and so is the term for $n^2$). The omitted lower order terms (which don't matter for \eqref{eqn:relasympmul}) depend on $\mu$ and also exhibit quasi-polynomiality.
This establishes \eqref{eqn:relasympmul} for $\lambda=\overline \Lambda_1$.

Concluding, for $\overline{\la{g}} = \lietype{G}_2$, we may take any $\lambda\in \overline{P}_+\backslash\{0\}$ in Theorem \ref{thm:ADElimits} part (2).

\section{Questions}

There are several questions for further study.

\begin{enumerate}

\item Assuming Conjecture \ref{conj:main}, the cases considered in the main Theorem \ref{thm:ADElimits} show that the limits of $\widehat{\jones}$ depend only on the coset $\lambda+\overline{Q}$ and not on the colour $\lambda$ itself. Thus, for a fixed torus knot, there are at most as many limits as the fundamental group $\overline{P}/\overline{Q}$.
This phenomenon appears to not have been discussed in the literature before. Is there an intrinsic, colour-independent (but $\overline{P}/\overline{Q}$-dependent) characterization of these limits? Does this same phenomenon hold for other knots and links?

\item What happens in the simply-laced case if $p,p'$ are not both $\geq h^\vee$? The earliest instance of this is the paper of Garoufalidis--Morton--Vuong \cite{GarMorVuo}, see also some specific conjectures from our paper \cite{Kan-torus}.

\item The estimate provided in Theorem \ref{thm:minexp} is quite crude, as explained in Remark \ref{rem:minexp}. Could one use the structure theory of the underlying lattice, Voronoi cells \cite{ConSlo-book}, etc., to get better estimates? Can these techniques be used to bypass representation-theoretic arguments?

\item  Why principal $\sW$ algebras? Is there an a priori reason that the coloured invariants of torus knots are related to $\sW$ algebras?

\item In the absence of a full proof of Conjecture \ref{conj:main}, at least some further families of examples would be desirable. This is being investigated \cite{KanAyy}.
\end{enumerate}


\providecommand{\oldpreprint}[2]{\textsf{arXiv:\mbox{#2}/#1}}\providecommand{\preprint}[2]{\textsf{arXiv:#1
		[\mbox{#2}]}}

\end{document}